\RequirePackage{fix-cm}
\documentclass[smallextended]{svjour3}       % onecolumn (second format)
\smartqed  % flush right qed marks, e.g. at end of proof
\usepackage{graphicx}
\usepackage{graphicx}
\usepackage{amsfonts}
\usepackage{amsmath}
\usepackage{amssymb}
\usepackage{amscd,multirow}
\usepackage{graphicx}
\usepackage{epstopdf}
\usepackage{subfigure}
\usepackage{subeqnarray,cases}
\usepackage{hyperref}
\usepackage[misc]{ifsym}
\usepackage[justification=centering]{caption}
\usepackage{cite}

\newtheorem{assumption}{Assumption}

\allowdisplaybreaks[4]

\usepackage[ruled]{algorithm2e}
\begin{document}

\title{Accelerated linearized alternating direction method of multipliers with Nesterov extrapolation \thanks{This work was supported by  the National Natural Science  Foundation of China (12171339), and Talent Introduction Project of Xihua University (Grant No. Z241102).}}

\author{Xin He  \and
        Nan-Jing Huang 	 \and
        Ya-Ping Fang
}


\institute{Xin He \at
             School of Science, Xihua University, Chengdu, Sichuan,  P.R. China \\
              \email{hexinuser@163.com}           %  \\
           \and
           \  Nan-Jing Huang \at
             Department of Mathematics, Sichuan University, Chengdu, Sichuan, P.R. China\\
           \email{njhuang@scu.edu.cn}
           \and
           Ya-Ping Fang \Letter  \at
           Department of Mathematics, Sichuan University, Chengdu, Sichuan, P.R. China \\
              \email{ypfang@scu.edu.cn}
}

\date{Received: date / Accepted: date}

\maketitle

\begin{abstract}
The alternating direction method of multipliers (ADMM) has found widespread use in  solving separable convex optimization problems. In this paper, by employing Nesterov extrapolation technique, we propose two families of accelerated linearized ADMMs for addressing two-block linearly constrained separable convex optimization problems where each block of the objective function exhibits a ``nonsmooth'' + ``smooth'' composite structure. Our proposed accelerated linearized ADMMs extend two classical Nesterov acceleration methods designed for unconstrained composite optimization problems to linearly constrained problems. These methods are capable of achieving non-ergodic convergence rates of $\mathcal{O}(1/k^2)$, provided that one block of the objective function exhibits strong convexity and the gradients of smooth terms are Lipschitz continuous. We show that the proposed methods can reduce  to  accelerated linearized augmented Lagrangian methods (ALMs) for solving one-block linearly constrained  convex optimization problems. By choosing different extrapolation parameters, we explore the relationship between the proposed methods and some existing accelerated methods. Numerical results are presented to validate the efficacy and reliability of the proposed algorithms.

\end{abstract}

\section{Introduction}
\subsection{Problem formulation and motivation}

In this paper, we address the  separable convex optimization problem with a composite structure
	\begin{equation}\label{ques}
				\min_{x\in  \mathbb{R}^{m},y\in  \mathbb{R}^{n}} \quad f(x)+g(y) \qquad s.t.  \  Ax+By=b,
	\end{equation}
where  $A\in\mathbb{R}^{p\times m}$, $B\in\mathbb{R}^{p\times n}$, $b\in\mathbb{R}^{p}$, $f(x)=f_1(x)+f_2(x)$, and $g(y)=g_1(y)+g_2(y)$. Here $f_1:\mathbb{R}^{m}\to\mathbb{R}\cup\{+\infty\}$ and $g_1:\mathbb{R}^{n}\to\mathbb{R}\cup\{+\infty\}$ are proper, closed and convex functions that may be nonsmooth, while $f_2:\mathbb{R}^{m}\to\mathbb{R}$ and $g_2:\mathbb{R}^{n}\to\mathbb{R}$ are smooth and convex.  It's worth noting that $f_2$ and $g_2$ can vanish in problem \eqref{ques}.  Problem \eqref{ques} finds applications in various fields, including  machine learning, image processing, computer vision, and distributed optimization (see \cite{LiML2019ACC,LiML2019ADMM,Boyd2011,Goldstein2014}).
	
Due to the demands of big data applications, first-order methods have gained significant popularity in recent years. However, naive gradient methods often exhibit slow convergence. Therefore, considerable effort has been devoted to accelerating various first-order methods. The origins of accelerated gradient methods date back to the 1980s when Nesterov \cite{Nesterov1983} introduced the  accelerated gradient method
\begin{equation}\label{al_Nest1983}
	\begin{cases}
		 y_{k} = x_{k}+\frac{t_k-1}{t_{k+1}}(x_{k}-x_{k-1}), \\
		x_{k+1} = y_k-s\nabla f(y_k),
	\end{cases}
\end{equation}
which was designed for solving the smooth convex programming $\min_x f(x)$,
where  $f:\mathbb{R}^{m}\to\mathbb{R}$ is a smooth convex function,  and the update scheme for  $\{t_k\}_{k\geq 1}$ is defined as:
\begin{eqnarray}\label{tk}
	t_{k+1}=\frac{1+\sqrt{1+4t_k^2}}{2}.
\end{eqnarray}
The  accelerated gradient method \eqref{al_Nest1983} exhibits a convergence rate of $f(x_k)- \min f = \mathcal{O}(1/k^2)$ when $f$ has an $L_f$-Lipschitz continuous gradient and $0<s\leq 1/L_{f}$.  Since 2008, as problems of increasing scale have emerged across various fields, researchers have extended the acceleration concept of the gradient method \eqref{al_Nest1983} to the composite convex optimization problem of the form 
\begin{equation}\label{tbl-f1}
\min_x f(x)+g(x),
\end{equation} 
where $f:\mathbb{R}^{m}\to\mathbb{R}$ is a smooth convex function and  $g:\mathbb{R}^{m}\to\mathbb{R}\cup\{+\infty\}$ is a proper, closed and convex function which may be nonsmooth. There are two well-known Nesterov's accelerated schemes for solving problem \eqref{tbl-f1}.  The  Nesterov's first scheme  
is presented as follows:
\begin{equation}\label{FISTA}
	\begin{cases}
		 y_{k} = x_{k}+\frac{t_k-1}{t_{k+1}}(x_{k}-x_{k-1}),\\
		x_{k+1} =\text{\bf Prox}_{s,g}(y_k-s\nabla f(x_k)).
	\end{cases}
\end{equation}
Here, $\text{\bf Prox}_{s,g}$ is the proximal operator for the function  $g$, defined as 
\[\text{\bf Prox}_{s,g}(y) = \mathop{\arg\min}_x g(x)+\frac{1}{2s}\|x- y\|^2.\]
When $t_k$ is updated under the rule \eqref{tk},  the  Nesterov's first scheme \eqref{FISTA} is known as the  fast iterative shrinkage-thresholding algorithm (FISTA) \cite{BeckIma}.
The  Nesterov's second scheme (refer to  \cite{Tseng2008} and \cite[Algorithm 2.3]{LiML2019ACC}) is presented as follows\footnote{In order to maintain the consistency of the extrapolation terms, we reorganize the algorithm.}:
\begin{equation}\label{NestSS}
	\begin{cases}
		 y_{k} = x_{k}+\frac{t_k-1}{t_{k+1}}(x_{k}-x_{k-1}),\\
		u_{k+1} = \mathop{\arg\min}_u g(u)+\langle \nabla f(y_k),u\rangle +\frac{1}{2st_{k+1}}\|u- u_k\|^2,\\
		x_{k+1} = \frac{t_{k+1}-1}{t_{k+1}}x_k+	\frac{1}{t_{k+1}}u_{k+1}	.
	\end{cases}
\end{equation}
When $g(x)=0$,  the  Nesterov's second scheme \eqref{NestSS}  is equivalent to the  Nesterov's first scheme \eqref{FISTA}. Both of these two Nesterov's accelerated schemes achieve the convergence rate of $(f+g)(x_k)- \min (f+g) = \mathcal{O}(1/k^2)$. Since then, numerous accelerated methods with different choices of $t_{k}$ have been developed for solving unconstrained problems (for references, see \cite{Nesterov2018,Attouch2018MP,Nesterov2013,Su2014,Attouch2018,LuoMP}).

Meanwhile, by incorporating Nesterov's extrapolation technique  \cite{Nesterov1983,Nesterov2018}, researches have made extensive efforts  to develop acclerated  ALMs for solving the one-block linearly constrained convex optimization problem
\begin{equation}\label{ques_one}
\min_{x\in \mathbb{R}^{m}} \quad f(x) \quad \text{s.t.} \ Ax=b,
\end{equation}
which is the case of problem \eqref{ques} with $B=0$ and $g(y)=0$. By introducing Nesterov's extrapolation to the dual sequence, several accelerated ALMs with $\mathcal{O}(1/k^2)$ convergence rates of the dual residual have been investigated in \cite{Huang2013,Kang2013,Kang2015,HeY2010}  under different assumptions. Xu \cite{Xu2017} generalized Nesterov's second scheme \eqref{NestSS} to a primal-dual framework, demonstrating $\mathcal{O}(1/k^2)$ non-ergodic convergence rates for both the objective residual and the feasibility violation by setting $t_{k}=\frac{k}{2}$. He et al. \cite{HeNA} introduced inertial accelerated primal-dual methods inspired by the second-order primal-dual dynamical system in \cite{HeSiam}. They applied Nesterov's extrapolation to both primal and dual variables, extending  Nesterov's first scheme \eqref{FISTA} with $t_{k} = 1+\frac{k-2}{\alpha-1}$ and established $\mathcal{O}(1/k^2)$ non-ergodic convergence rates for the objective residual and the feasibility violation. Additionally, Tran-Dinh et al. \cite{Yurtsever2015,Tran2018,Patrascu2017} developed an accelerated primal-dual framework that achieves non-ergodic convergence rates. He et al. \cite{HeHess,HeAuto}, Luo \cite{LuoEsiam,Luo2021Accer}, and Bo{\c t} et al. \cite{BotMP2022} proposed different accelerated primal-dual methods with fast convergence rates through time discretization of primal-dual dynamical systems. For more accelerated ALMs for problem \eqref{ques_one},  we refer the reader to \cite{Sabach2022,Zhang2022,Xu2021MP,Tao2017,Lin2023,Zhu2023}.

For the two-block separable problem \eqref{ques}, ADMM \cite{Gabay1976} is a widely adopted method. It alternates between updating the primal variables $x$ and $y$, and subsequently updates the dual variable $\lambda$. The iteration scheme of ADMM can be described as follows:
\begin{equation*}\label{ADMM}
	\begin{cases}
		x_{k+1} =  \mathop{\arg\min}_x \mathcal{L}(x,y_k,\lambda_k)+\frac{\sigma}{2}\|Ax+By_{k}-b\|^2,\\
        y_{k+1} =  \mathop{\arg\min}_y \mathcal{L}(x_{k+1},y,\lambda_k)+\frac{\sigma}{2}\|Ax_{k+1}+By-b\|^2,\\
        \lambda_{k+1} =\lambda_k+\sigma(Ax_{k+1}+By_{k+1}-b).
	\end{cases}
\end{equation*}
Here  $\sigma>0$ and  $\mathcal{L}$ is the Lagrangian function of problem \eqref{ques}, defined as
 \begin{equation}\label{lagran}
 	\mathcal{L}(x,y,\lambda)= f(x)+g(y)+\langle \lambda,Ax+By-b\rangle.
 \end{equation}
There are several variants of ADMM, including those with proximal preconditioning \cite{Chen2022,Chen2015}, linearization \cite{Ouyang,Goldfarb}, inexactness \cite{Adona2020,Bai2022}, and symmetrization \cite{Hepr2014,Davis2017}. In the convex setting of functions $f$ and $g$,  ergodic $\mathcal{O}(1/k)$  convergence rates of ADMM and its variants  were derived in \cite{HeY2012Siam,Xu2017} and non-ergodic $\mathcal{O}(1/k)$  convergence rates were established in \cite{HeY2015NM,Li2019}.  
Incorporating Nesterov extrapolation to ADMMs, various accelerated ADMMs were developed in the literature. Naturally,  $\mathcal{O}(1/k^2)$ is a desired  convergence rate for an accelerated ADMM with Nesterov extrapolation.  But, it still remains open whether or not an ADMM-type method can acchieve the $\mathcal{O}(1/k^2)$  under a general setting where the objective functions are assumed to be only convex.  Some accelerated ADMMs  with  $\mathcal{O}(1/k^2)$ convergence rates have been investigated  under special assumptions on the models. Tian and Yuan \cite{Tian2019} proposed an accelerated ADMM for problem \eqref{ques} with $B= -Id$ and $b=0$, where $Id$ denotes the identity map,  and proved an ergodic $\mathcal{O}(1/k^2)$ convergence rate when $g$ has a Lipschitz continuous gradient and $A$ has a full column rank.  Recently, researchers have focused on developing accelerated ADMM under the semi-strong convexity assumption in the sense that one of the functions $f$ and $g$ is strongly convex and another is only convex. Xu \cite{Xu2017} proposed an accelerated linearized ADMM for solving problem \eqref{ques} with $g=g_1+g_2$. By adapting the parameters, they proved an ergodic $\mathcal{O}(1/k^2)$ convergence rate when $g$ is strongly convex and $g_2$ possesses a Lipschitz continuous gradient. Under the semi-strong convexity assumption,  Tran-Dinh et al. \cite{Tran2019,Tran2020} introduced new accelerated ADMM methods for problem \eqref{ques} with  non-ergodic $\mathcal{O}(1/k^2)$ convergence rates. Note that in the methods proposed in \cite{Tran2019,Tran2020}, three subproblems are solved in each iteration, one more than the classical ADMM. Zhang et al. \cite{Zhang2023} presented a unified prediction-correction framework to accelerate Lagrangian-based methods and established both ergodic $\mathcal{O}(1/k^2)$ and non-ergodic $\mathcal{O}(1/k^2)$ convergence rates for accelerated ADMM when the objective function is semi-strongly convex. The subproblem in \cite{Zhang2023} cannot involve the proximal calculation of strongly convex functions, and they do not address problem \eqref{ques} when $f$ and $g$ have a composite structure. Very recently, Luo and Zhang \cite{LuoSco} proposed a unified differential equation solver approach for problem \eqref{ques}. By utilizing different time discretizations of an inertial primal-dual dynamical system, they developed a class of accelerated ADMMs which exhibit non-ergodic optimal mixed-type convergence rates under both convex and semi-strongly convex assumptions. It is worth mentioning that multiple adaptive parameters of the algorithm considered in  \cite{LuoSco} must satisfy certain complex equality conditions to attain the non-ergodic convergence rate. These conditions play a crucial role in the proof and operation but introduce challenges in parameter selection during actual numerical calculations.

 \subsection{ Notations and assumption} 

Let $\langle \cdot, \cdot \rangle$ and $\|\cdot\|$ represent the inner product and the Euclidean norm, respectively. We denote 
$\Omega\in \mathbb{R}^{m}\times\mathbb{R}^{n}\times\mathbb{R}^{p}$ as 
the set of saddle points  for the Lagrangian function $\mathcal{L}$ defined in \eqref{lagran}. For any $(x^*,y^*,\lambda^*)\in\Omega$, we have
\begin{equation}\label{saddle}
	\mathcal{L}(x^*,y^*,\lambda) \leq \mathcal{L}(x^*,y^*,\lambda^*)\leq \mathcal{L}(x,y,\lambda^*), \quad \forall x,y,\lambda \in \mathbb{R}^{m}\times\mathbb{R}^{n}\times\mathbb{R}^{p},
\end{equation}
$(x^*,y^*)$ is a solution of problem \eqref{ques},  and
\[Ax^*+By^*-b=0.\]

For a  function  $F:\mathbb{R}^{n}\to\mathbb{R}\cup\{+\infty\}$, the domain of $F$ is defined as $dom(F)=\{x\in\mathbb{R}^n | F(x)<+\infty\}$. We say that  $F$ is proper if $dom(F) \neq \emptyset$, and that  $F$ is closed if $F(x) \leq \liminf_{y\to x}F(y)$ and $dom(F)$ is closed. For a proper, closed and convex function $F:\mathbb{R}^{n}\to\mathbb{R}$,  the subdifferential $\partial F$ of $F$ is defined by
 \[\partial F(x) = \{\omega \in\mathbb{R}^n | F(y)\geq F(x)+\langle \omega,y-x \rangle,\quad \forall y\in\mathbb{R}^n\}.\]
We denote $\widetilde{\nabla} F(x)\in\partial F(x)$ as a subgradient of $F$ at $x$. When $F$ is smooth, $\widetilde{\nabla} F(x)=\nabla F(x)$ is the gradient of $F$ at $x$.  

   To establish fast convergence rates for the proposed  methods, our assumption regarding the model \eqref{ques} can be summarized as follows.

\begin{assumption}\label{ass}
	 Suppose that $f(x)=f_1(x)+f_2(x)$ and $g(y)=g_1(y)+g_2(y)$, where $f_1$ is a proper, closed and convex function, $f_2$ is convex and  has a Lipschitz continuous gradient with constant $L_{f_2}$;  $g_1$ is a proper, closed and $\mu_g$-strongly convex function, $g_2$ is convex and has a Lipschitz continuous gradient with constant $L_{g_2}$; the saddle points set $\Omega$ is nonempty.
\end{assumption}
With the above assumption, we have 
\begin{eqnarray}\label{strongC}
	g_1(y_2) - g_1(y_1)-\langle  \widetilde{\nabla} g_1(y_1), y_2-y_1\rangle\geq \frac{\mu_g}{2}\|y_2-y_1\|^2 
\end{eqnarray}
for any $y_1,y_2 \in\mathbb{R}^n, \widetilde{\nabla} g_1(y_1)\in\partial g_1(y_1)$,
and
\begin{equation*}
f_2(x_2) -f_2(x_1)-\langle \nabla f_2(x_1),x_2-x_1\rangle \leq \frac{L_{f_2}}{2}\|x_2-x_1\|^2, \quad \forall x_1,x_2\in\mathbb{R}^m.
\end{equation*}
This together with the convexity of $f_2$ implies
\begin{eqnarray}\label{lip_f2}
	\langle\nabla f_2(\bar{x}), x_2-x_1\rangle &=& \langle\nabla f_2(\bar{x}), x_2-\bar{x}\rangle +\langle\nabla f_2(\bar{x}), \bar{x}-x_1\rangle\nonumber\\
	&\geq & f_2(x_2)-f_2(\bar{x})-\frac{L_{f_2}}{2}\|x_2-\bar{x}\|^2+f_2(\bar{x})-f_2(x_1)\\
	&=& f_2(x_2)-f_2(x_1)-\frac{L_{f_2}}{2}\|x_2-\bar{x}\|^2\nonumber
\end{eqnarray}
for any $x_1,x_2,\bar{x}$. Similarly,  we have
\begin{equation}
	\langle\nabla g_2(\bar{y}), y_2-y_1\rangle \geq g_2(y_2)-g_2(y_1)-\frac{L_{g_2}}{2}\|y_2-\bar{y}\|^2
\end{equation}
holds  for any $y_1,y_2,\bar{y}$.

\subsection{ Main contributions} 

We summarize our main contributions below.

$\bullet$ We propose two families of accelerated linearized ADMMs for solving problem \eqref{ques} with a composite structure. These two families of acceleration methods can be seen as generalizations of Nesterov's first scheme  \eqref{FISTA} and Nesterov's second scheme \eqref{NestSS}, respectively. Some of the proposed accelerated linearized ADMMs allow for the linearization of the augmented term for updating variable $y$, which involves only the proximal calculation of $g_1$. Under Assumption \ref{ass}, we demonstrate that the Lagrangian residual, the objective residual, and the feasibility violation of all the proposed methods enjoy non-ergodic $\mathcal{O}(1/t_k^2)$ convergence rates when $t_{k+1}^2\leq t_k^2 + t_{k+1}$, and can achieve a decay rate of $\mathcal{O}(1/k^2)$ at most. 

 $\bullet$ We demonstrate that the proposed methods can be reduced to accelerated linearized ALMs for solving the one-block linearly constrained problem \eqref{ques_one} with a composite structure. Under the assumption of convexity for objective function, some methods can attain a non-ergodic decay rate of $\mathcal{O}(1/k^2)$. Under the assumption of strong convexity for $f$, some methods can achieve a non-ergodic decay rate of $\mathcal{O}(1/k^2)$ with the proximal calculation of subproblems, allowing for the linearization of the augmented term. In addition, the accelerated  linearized ALMs can be reduced to Nesterov's accelerated schemes \eqref{FISTA} and \eqref{NestSS}  for solving unconstrained optimization problems.

$\bullet$ We discuss the relationship between our accelerated linearized ADMM and some existing acceleration algorithms, including accelerated linearized ALMs in \cite{Xu2017,HeNA}, accelerated ALMs in \cite{BotMP2022,Sabach2022}, and accelerated forward-backward methods in \cite{BeckIma,Tseng2008,Attouch2018,Chambolle}. We also conduct a series of numerical experiments to validate the effectiveness of the proposed algorithms.

\subsection{Outlines}

The structure of this paper is as follows. In Section 2, we propose a family of accelerated linearized ADMMs for solving \eqref{ques} by extending Nesterov's first scheme \eqref{FISTA} to the ADMM framework  and discuss  convergence rate analysis of the proposed methods. Section 3  is devoted to the study on  accelerated linearized ADMMs based on  Nesterov's second scheme for solving \eqref{ques}.  In Section 4, we discuss relationship between the proposed algorithms and  some existing accelerated methods. The numerical results are presented in Section 5, and we provide concluding remarks in Section 6.

\section{Accelerated Linearized ADMM Based on Nesterov's First Scheme}\label{sec_3}
\subsection{Formulation of algorithm}

In this part, we attempt to extend Nesterov's first scheme \eqref{FISTA}  to the ADMM framework. By applying the optimality condition and introducing an auxiliary variable $u_{k+1}=x_{k+1}+(t_{k+1}-1)(x_{k+1}-x_k)$, algorithm \eqref{FISTA} can be recast as the following equivalent scheme:
\begin{numcases}{}\label{eq_fista}
        {y}_k = x_k + \frac{t_k-1}{t_{k+1}}(x_k-x_{k-1}),\nonumber \\
		u_{k+1}-u_k = -{s t_{k+1}}(\widetilde{\nabla} g(\boldsymbol{x_{k+1}})+\nabla f(y_k)),\label{dis_Nest2} \\
		x_{k+1} = \frac{1}{t_{k+1}}u_{k+1}+ \frac{t_{k+1}-1}{t_{k+1}}x_{k}.\nonumber
 \end{numcases}

Building on this, we propose the following scheme for problem \eqref{ques}:
\begin{subequations}\label{disB}
\begin{numcases}{}
        (\bar{x}_k, \bar{y}_k) = (x_k,y_k) + \frac{t_k-1}{t_{k+1}}[(x_k,y_k)-(x_{k-1},y_{k-1})],\label{disB1}\\
		u_{k+1}-u_k = -\boldsymbol{\alpha t_{k+1}}(\widetilde{\nabla} f_1(\boldsymbol{x_{k+1}})+\nabla f_2(\bar{x}_k)+A^T\bar{\lambda}_{k+1}),\label{disB2}\\
		v_{k+1}-v_k = -\boldsymbol{\frac{\beta}{t_{k+1}}}(\widetilde{\nabla} g_1(\boldsymbol{y_{k+1}})+\nabla g_2(\bar{y}_k) +B^T\hat{\lambda}_{k+1}+\boldsymbol{\mu_g(v_{k+1}-y_{k+1})}),\label{disB3}\\
		(x_{k+1},y_{k+1}) = \frac{1}{t_{k+1}}(u_{k+1},v_{k+1})+ \frac{t_{k+1}-1}{t_{k+1}}(x_{k},y_k),\label{disB4}\\
		\lambda_{k+1} = \lambda_k + \gamma t_{k+1}(Au_{k+1}+Bv_{k+1}-b),\label{disB5}
 \end{numcases}
\end{subequations}
where $x_1=x_0\in\mathbb{R}^m, y_1=y_0\in\mathbb{R}^n$, $u_1 = x_1$, $v_1=y_1$, $\alpha>0$, $\beta>0$, $\gamma>0$, and $t_{k}\geq 1$ is a nondecreasing sequence.

It can be observed from \eqref{disB2} and  \eqref{disB4} that the update of the primal variable $x_k$ is same as  \eqref{eq_fista}. Due to the alternating update between the primal variables $x$ and $y$, we choose the scaling $\frac{\beta}{t_{k+1}}$ in \eqref{disB3} instead of ${\beta}{t_{k+1}}$. 
Remark that choosing the scaling  as ${\beta t_{k+1}}$, would require setting $\bar{\lambda}_{k+1}=\hat{\lambda}_{k+1}=\lambda_{k+1}$ to maintain fast convergence. However, this would make  $u_{k+1}$ and $v_{k+1}$  coupled with each other, and cannot  yield a useful algorithm. In this scenario, the term $\mu_g(v_{k+1}-y_{k+1})$ is introduced  into \eqref{disB3}, which is motivated by the accelerated first-order methods considered  in \cite{LuoMP}, and it plays a critical element in our convergence analysis. We  select $\bar{\lambda}_{k+1}$  and  $\hat{\lambda}_{k+1}$ as follows:
\begin{eqnarray*}
	 \bar{\lambda}_{k+1} = \lambda_k +\gamma t_{k+1}(Au_{k+1}+Bv_{k}-b);\qquad \hat{\lambda}_{k+1} = {\lambda}_{k+1} \text{ or } \bar{\lambda}_{k+1} .
\end{eqnarray*}
By applying the optimality condition and noting that
\[ u_{k+1}=t_{k+1}x_{k+1}-(t_{k+1}-1)x_k,\qquad u_{k+1}-u_k=t_{k+1}(x_{k+1}-\bar{x}_k),\]
we can equivalently represent equation \eqref{disB2} as:
\begin{eqnarray}\label{updata_uB1}
&&	x_{k+1} = \mathop{\arg\min}_x\left(f_1(x)+\langle A^T\lambda_k+\nabla f_2(\bar{x}_k),x\rangle+\frac{\gamma t^2_{k+1}}{2}\|A(x-x_k)+\right.\\
 &&    \left.\qquad\qquad\qquad \qquad\frac{1}{t_{k+1}}(Ax_k+Bv_k-b)\|^2+\frac{1}{2\alpha}\left\|x-\bar{x}_k\right\|^2\right).\nonumber
\end{eqnarray}
Similarly, when $\hat{\lambda}_{k+1} ={\lambda}_{k+1}$,  equation \eqref{disB3} is equivalent to
\begin{eqnarray}\label{updata_vB1}
	y_{k+1}& =& \mathop{\arg\min}_y\left(g_1(y)+\langle B^T\lambda_k+\nabla g_2(\bar{y}_k),y\rangle\right.\nonumber\\
	&&\left.+\frac{1}{2\eta_k}\left\|y-\bar{y}_k+\eta_k\mu_g(t_{k+1}-1)(\bar{y}_k-y_k)\right\|^2\right.\\
   &&\left.+\frac{\gamma t^2_{k+1}}{2}\|B(y-y_k)+\frac{1}{t_{k+1}}(Au_{k+1}+By_k-b)\|^2\right), \nonumber
\end{eqnarray}
where $\eta_k = \frac{\beta}{t_{k+1}^2+\beta\mu_g(t_{k+1}-1)}$. In this case, we can rewrite scheme \eqref{disB} as Algorithm \ref{al_2A}.

When $\hat{\lambda}_{k+1} =\bar{\lambda}_{k+1}$,  equation \eqref{disB3} is equivalent to
\begin{equation}\label{updata_vB2}
     y_{k+1} ={\bf Prox}_{{\eta_k},g_1}\left(\bar{y}_k -\eta_k(\mu_g(t_{k+1}-1)(\bar{y}_k-y_k)+B^T\bar{\lambda}_{k+1}+\nabla g_2(\bar{y}_k))\right)
\end{equation}
with $\eta_k = \frac{\beta}{t_{k+1}^2+\beta\mu_g(t_{k+1}-1)}$.
In this case, we can rewrite  scheme \eqref{disB} as Algorithm \ref{al_2B}.

\begin{algorithm}
        \caption{Accelerated Linearized ADMM-Nesterov's First Scheme (I)}
        \label{al_2A}
        {\bf Initialize:} Let $x_1=x_0=u_1\in dom(f), y_1=y_0=v_1\in dom(g), \lambda_1\in\mathbb{R}^p$, $\alpha>0,\beta>0,\gamma>0$, $t_1\geq 1$. \\
        \For{$k = 1, 2,\cdots$}{
        $(\bar{x}_k, \bar{y}_k) = (x_k,y_k) + \frac{t_k-1}{t_{k+1}}[(x_k,y_k)-(x_{k-1},y_{k-1})]$.\\
       $x_{k+1} = \mathop{\arg\min}_x\left(f_1(x)+\langle A^T\lambda_k+\nabla f_2(\bar{x}_k),x\rangle+\frac{\gamma t^2_{k+1}}{2}\|A(x-x_k)+\frac{1}{t_{k+1}}(Ax_k+Bv_k-b)\|^2\right.$\\
     $\left.\qquad\qquad\qquad \qquad+\frac{1}{2\alpha}\left\|x-\bar{x}_k\right\|^2\right)$.\\
      $u_{k+1}=x_{k+1}+(t_{k+1}-1)(x_{k+1}-x_k)$.\\
      Let $\eta_k = \frac{\beta}{t_{k+1}^2+\beta\mu_g(t_{k+1}-1)}$.\\
      $y_{k+1} = \mathop{\arg\min}_y\left(g_1(y)+\langle B^T\lambda_k+\nabla g_2(\bar{y}_k),y\rangle+\frac{1}{2\eta_k}\left\|y-\bar{y}_k+\eta_k\mu_g(t_{k+1}-1)(\bar{y}_k-y_k)\right\|^2\right.$\\
  $ \left.\qquad\qquad\qquad \qquad +\frac{\gamma t^2_{k+1}}{2}\|B(y-y_k)+\frac{1}{t_{k+1}}(Au_{k+1}+By_k-b)\|^2\right) $.\\
          $v_{k+1}=y_{k+1}+(t_{k+1}-1)(y_{k+1}-y_k)$.\\
		$\lambda_{k+1} = \lambda_k + \gamma t_{k+1}(Au_{k+1}+Bv_{k+1}-b)$.
}
\end{algorithm}

\begin{algorithm}
        \caption{Accelerated Linearized ADMM-Nesterov's First Scheme (II)}
        \label{al_2B}
         {\bf Initialize:} Let $x_1=x_0=u_1\in dom(f), y_1=y_0=v_1\in dom(g), \lambda_1\in\mathbb{R}^p$, $\alpha>0,\beta>0,\gamma>0$, $t_1\geq 1$. \\
        \For{$k = 1, 2,\cdots$}{
        $(\bar{x}_k, \bar{y}_k) = (x_k,y_k) + \frac{t_k-1}{t_{k+1}}[(x_k,y_k)-(x_{k-1},y_{k-1})]$.\\
        $x_{k+1} = \mathop{\arg\min}_x\left(f_1(x)+\langle A^T\lambda_k+\nabla f_2(\bar{x}_k),x\rangle+\frac{\gamma t^2_{k+1}}{2}\|A(x-x_k)+\frac{1}{t_{k+1}}(Ax_k+Bv_k-b)\|^2\right.$\\
     $\left.\qquad\qquad\qquad \qquad+\frac{1}{2\alpha}\left\|x-\bar{x}_k\right\|^2\right)$.\\
      $u_{k+1}=x_{k+1}+(t_{k+1}-1)(x_{k+1}-x_k)$.\\
      $\bar{\lambda}_{k+1} =\lambda_k+ \gamma t_{k+1}(Au_{k+1}+Bv_{k}-b)$,\\
       Let $\eta_k = \frac{\beta}{t_{k+1}^2+\beta\mu_g(t_{k+1}-1)}$.\\
      $y_{k+1} ={\bf Prox}_{{\eta_k},g_1}\left(\bar{y}_k -\eta_k(\mu_g(t_{k+1}-1)(\bar{y}_k-y_k)+B^T\bar{\lambda}_{k+1}+\nabla g_2(\bar{y}_k))\right)$.\\
      $v_{k+1}=y_{k+1}+(t_{k+1}-1)(y_{k+1}-y_k)$.\\
		$\lambda_{k+1} = \lambda_k + \gamma t_{k+1}(Au_{k+1}+Bv_{k+1}-b)$.
}
\end{algorithm}

\subsection{A single-step analysis}
Below, we present a single-iteration analysis for the sequence that satisfies \eqref{disB}. The results will be employed in the convergence analysis of Algorithm \ref{al_2A} and Algorithm \ref{al_2B}. To begin, we construct the energy sequence $\{\mathcal{E}_k\}_{k\geq 1}$ as:
\begin{eqnarray*}
	\mathcal{E}_k =I_k^1 +I_k^2+I_k^3+I_k^4
\end{eqnarray*}
with 
\begin{numcases}{}
	 I_k^1 = t^2_k(\mathcal{L}(x_k,y_k,\lambda^*)-\mathcal{L}(x^*,y^*,\lambda^*)),\nonumber \\
	 I_k^2 = \frac{1}{2\alpha}\|u_k-x^*\|^2,\nonumber \\
	 I_k^3 = \frac{t^2_{k+1}}{2\beta}\|v_k-y^*\|^2,\label{energy_sub}  \\
	 I_k^4 = \frac{1}{2\gamma}\|\lambda_k-\lambda^*\|^2. \nonumber 
\end{numcases}
It follows from \eqref{saddle} that $\mathcal{E}_k\geq 0$ for all $k\geq 1$. Before we proceed to estimate the energy sequence $\{\mathcal{E}_k\}_{k\geq 1}$, we recall the following equality:
\begin{equation}\label{eq_know}
	 \frac{1}{2}\|x\|^2-\frac{1}{2}\|y\|^2=\langle x,x-y\rangle-\frac{1}{2}\|x-y\|^2
\end{equation}
for any $x,y,z \in\mathbb{R}^n$, which  will be used repeatedly in our analysis.

\begin{lemma}\label{le_B1}
Suppose that Assumption \ref{ass} holds. Let $\{(x_k,y_k,\lambda_k,u_k,v_k,\bar{\lambda}_k, \hat{\lambda}_k)\}_{k\geq 1}$ be  the sequence satisfying \eqref{disB}, with  $\alpha,\beta,\gamma>0$ and $\{t_k\}_{k\geq 1}\subseteq [1,+\infty)$ such that
\[ t_{k+1}^2\leq t^2_k + t_{k+1},\qquad k\geq 1.\]
Define the energy sequence $\mathcal{E}_k$ as in \eqref{energy_sub}.  Then  for any $(x^*,y^*,\lambda^*)\in\Omega$,  the following conclusions hold:
\begin{eqnarray*}
	\mathcal{E}_{k+1}-\mathcal{E}_k &\leq& \frac{t^2_{k+2}-t^2_{k+1}-\beta\mu_g t_{k+1}}{2\beta}\|v_{k+1}-y^*\|^2-\frac{1-\alpha L_{f_2}}{2\alpha}\|u_{k+1}-u_k\|^2\\
	&&-\frac{t^2_{k+1}-\beta L_{g_2}}{2\beta}\|v_{k+1}-v_k\|^2 -\frac{\mu_gt_{k+1}^2(t_{k+1}-1)}{2}\|y_{k+1}-y_k\|^2\\
	&&+\frac{1}{2\gamma}\|\lambda_{k+1}-\bar{\lambda}_{k+1}\|^2+t_{k+1}\langle B(v_{k+1}-y^*),\bar{\lambda}_{k+1}- \hat{\lambda}_{k+1}\rangle.
\end{eqnarray*}
\end{lemma}
\begin{proof}
Since  $t_{k+1}^2\leq t^2_k + t_{k+1}$, we have
\begin{eqnarray*}\label{eq_le2_1}
	I_{k+1}^1-I_k^1 &=& t^2_{k+1}(\mathcal{L}(x_{k+1},y_{k+1},\lambda^*)-\mathcal{L}(x^*,y^*,\lambda^*)) - t^2_k(\mathcal{L}(x_k,y_k,\lambda^*)-\mathcal{L}(x^*,y^*,\lambda^*))\nonumber\\
	&= &(t_{k+1}(t_{k+1}-1)-t^2_k)(\mathcal{L}(x_k,y_k,\lambda^*)-\mathcal{L}(x^*,y^*,\lambda^*))\nonumber\\
	&& + t^2_{k+1}(\mathcal{L}(x_{k+1},y_{k+1},\lambda^*)-\mathcal{L}(x^*,y^*,\lambda^*))\\
	&& + t_{k+1}(t_{k+1}-1)(\mathcal{L}(x_k,y_k,\lambda^*)- \mathcal{L}(x^*,y^*,\lambda^*))\nonumber\\
	&\leq & t_{k+1}(\mathcal{L}(x_{k+1},y_{k+1},\lambda^*)-\mathcal{L}(x^*,y^*,\lambda^*))\nonumber\\
	&& + t_{k+1}(t_{k+1}-1)(\mathcal{L}(x_{k+1},y_{k+1},\lambda^*)-\mathcal{L}(x_k,y_k,\lambda^*)).\nonumber
\end{eqnarray*}	
From  \eqref{disB4}, we can derive
\begin{equation*}\label{eq_le1_3}
	u_{k+1} = x_{k+1} +(t_{k+1}-1)(x_{k+1}-x_{k})
\end{equation*}
and 
\begin{equation*}\label{eq_le1_4}
	u_{k+1}-u_{k} = t_{k+1}(x_{k+1}-x_k)- (t_{k}-1)(x_{k}-x_{k-1}) =  t_{k+1}(x_{k+1}-\bar{x}_k).
\end{equation*}
By exploiting the convexity of $f_1$,  we have
\begin{eqnarray*}
&&	\langle u_{k+1}-x^*, \widetilde{\nabla} f_1({x}_{k+1})\rangle = \langle x_{k+1}-x^*,\widetilde{\nabla} f_1({x}_{k+1})\rangle+(t_{k+1}-1)\langle x_{k+1}-x_k, \widetilde{\nabla} f_1({x}_{k+1})\rangle\\
	&&\qquad\quad\geq  f_1(x_{k+1})-f_1(x^*)+(t_{k+1}-1)(f_1(x_{k+1})-f_1(x_k)).
\end{eqnarray*}
Since $f_2$ is convex and  has a  Lipschitz continuous gradient with constant $L_{f_2}$, it follows from \eqref{lip_f2} that
\begin{eqnarray*}\label{eq_le1_5}
&&	\langle u_{k+1}-x^*, {\nabla} f_2(\bar{x}_{k})\rangle = \langle x_{k+1}-x^*, {\nabla} f_2(\bar{x}_{k})\rangle+(t_{k+1}-1)\langle x_{k+1}-x_k, {\nabla} f_2(\bar{x}_{k})\rangle \nonumber\\
	&&\qquad\quad\geq  f_2(x_{k+1})-f_2(x^*)+(t_{k+1}-1)(f_2(x_{k+1})-f_2(x_k))-\frac{L_{f_2} t_{k+1}}{2}\|x_{k+1}-\bar{x}_k\|^2\\
	&&\qquad\quad=  f_2(x_{k+1})-f_2(x^*)+(t_{k+1}-1)(f_2(x_{k+1})-f_2(x_k))-\frac{L_{f_2}}{2t_{k+1}}\|u_{k+1}-u_k\|^2. \nonumber
\end{eqnarray*}
This, combined with  \eqref{disB2} and \eqref{eq_know} ,  leads to
\begin{eqnarray}\label{eq_le2_2}
	I_{k+1}^2-I_k^2 &=& \frac{1}{2\alpha}(\|u_{k+1}-x^*\|^2-\|u_{k}-x^*\|^2)\nonumber\\
	&=& \frac{1}{\alpha}\langle u_{k+1}-x^*, u_{k+1}-u_k\rangle-\frac{1}{2\alpha}\|u_{k+1}-u_k\|^2\nonumber\\
	&=& -t_{k+1} \langle u_{k+1}-x^*, \widetilde{\nabla} f_1(x_{k+1})+{\nabla} f_2(\bar{x}_{k})+A^T\lambda^*\rangle\nonumber\\
	&&-t_{k+1} \langle u_{k+1}-x^*,A^T(\bar{\lambda}_{k+1}-\lambda^*)\rangle-\frac{1}{2\alpha}\|u_{k+1}-u_k\|^2\nonumber\\
	&\leq & - t_{k+1}(f(x_{k+1})-f(x^*)+\langle \lambda^*,A(x_{k+1}-x^*)\rangle\\
	&& -t_{k+1}(t_{k+1}-1)(f(x_{k+1})-f(x_k)+\langle \lambda^*,A(x_{k+1}-x_k)\rangle)\nonumber\\
	&&-t_{k+1} \langle u_{k+1}-x^*,A^T(\bar{\lambda}_{k+1}-\lambda^*)\rangle-\frac{1-\alpha L_{f_2}}{2\alpha}\|u_{k+1}-u_k\|^2.\nonumber
\end{eqnarray}	
Similarly, we have
\begin{eqnarray*}
	\langle v_{k+1}-y^*,  \widetilde{\nabla} g_1({y}_{k+1})\rangle &\overset{\eqref{strongC}}{\geq}&  g_1(y_{k+1})-g_1(y^*)+(t_{k+1}-1)(g_1(y_{k+1})-g_1(y_k))\\
	&&  +\frac{\mu_g}{2}(\|y_{k+1}-y^*\|^2+(t_{k+1}-1)\|y_{k+1}-y_k\|^2).
\end{eqnarray*}
and 
\begin{eqnarray*}\label{eq_le1_7}
&&	\langle v_{k+1}-y^*, {\nabla} g_2(\bar{y}_{k})\rangle \\
	&&\qquad\quad\geq  g_2(y_{k+1})-g_2(y^*)+(t_{k+1}-1)(g_2(y_{k+1})-g_2(y_k))-\frac{L_{g_2}}{2t_{k+1}}\|v_{k+1}-v_k\|^2.\nonumber
\end{eqnarray*}
It follows from  \eqref{eq_know} and  $v_{k+1}-y_{k+1}=(t_{k+1}-1)(y_{k+1}-y_k)$ that
\begin{eqnarray*}
	\langle v_{k+1}-y^*, v_{k+1}-y_{k+1}\rangle = \frac{1}{2}(\|v_{k+1}-y^*\|^2+\| v_{k+1}-y_{k+1}\|^2-\|y_{k+1}-y^*\|^2)\\
	 = \frac{1}{2}(\|v_{k+1}-y^*\|^2+(t_{k+1}-1)^2\|y_{k+1}-y_k\|^2-\|y_{k+1}-y^*\|^2).
\end{eqnarray*}
Then,  by incorporating \eqref{disB3}, we can express
\begin{eqnarray}\label{eq_le2_3}
	&&I_{k+1}^3-I_k^3 = \frac{t^2_{k+2}}{2\beta}\|v_{k+1}-y^*\|^2-\frac{t^2_{k+1}}{2\beta}\|v_{k}-y^*\|^2\nonumber\\
	&&\qquad =\frac{t^2_{k+2}-t^2_{k+1}}{2\beta}\|v_{k+1}-y^*\|^2+ \frac{t^2_{k+1}}{\beta}(\langle v_{k+1}-y^*, v_{k+1}-v_k\rangle-\frac{1}{2}\|v_{k+1}-v_k\|^2)\nonumber\\
  &&\qquad	=\frac{t^2_{k+2}-t^2_{k+1}}{2\beta}\|v_{k+1}-y^*\|^2\nonumber \\
  &&\qquad\quad-t_{k+1} \langle v_{k+1}-y^*, \widetilde{\nabla} g_1(y_{k+1}) + {\nabla} g_2(\bar{y}_{k})+B^T{\lambda}^*+\mu_g(v_{k+1}-y_{k+1})\rangle\\
	&&\qquad\quad -t_{k+1} \langle v_{k+1}-y^*,B^T(\hat{\lambda}_{k+1}-\lambda^*)\rangle-\frac{t^2_{k+1}}{2\beta}\|v_{k+1}-v_k\|^2\nonumber\\
	&&\qquad\leq \frac{t^2_{k+2}-t^2_{k+1}-\beta\mu_g t_{k+1}}{2\beta}\|v_{k+1}-y^*\|^2-\frac{\mu_gt_{k+1}^2(t_{k+1}-1)}{2}\|y_{k+1}-y_k\|^2\nonumber\\
	&& \qquad \quad -t_{k+1}(g(y_{k+1})-g(y^*)+\langle \lambda^*,B(y_{k+1}-y^*)\rangle)\nonumber\\
	&& \qquad\quad-t_{k+1}(t_{k+1}-1)(g(y_{k+1})-g(y_k)+\langle \lambda^*,B(y_{k+1}-y_k)\rangle)\nonumber\\
	&& \qquad\quad -t_{k+1} \langle v_{k+1}-y^*,B^T(\hat{\lambda}_{k+1}-\lambda^*)\rangle-\frac{t^2_{k+1}-\beta L_{g_2}}{2\beta}\|v_{k+1}-v_k\|^2.\nonumber
\end{eqnarray}
Since $(x^*,y^*,\lambda^*)\in\Omega$ and $Ax^*+By^*-b=0$, it  follows from  \eqref{disB5} that
\begin{equation*}\label{eq_le1_9}
	\lambda_{k+1} - \lambda_k = \gamma t_{k+1}(Au_{k+1}+Bv_{k+1}-b) =\gamma t_{k+1}(A(u_{k+1}-x^*)+B(v_{k+1}-y^*)). 
\end{equation*}
This, combined with \eqref{eq_know}, results in
\begin{eqnarray}\label{eq_le1_10}
	I_{k+1}^4-I_k^4 &=& \frac{1}{2\gamma}(\|\lambda_{k+1}-\lambda^*\|^2-\|\lambda_{k}-\lambda^*\|^2)\nonumber\\
	&=& \frac{1}{\gamma}\langle \lambda_{k+1}-\lambda^*, \lambda_{k+1}-\lambda_k\rangle-\frac{1}{2\gamma}\|\lambda_{k+1}-\lambda_k\|^2\\
	&=& t_{k+1} \langle \lambda_{k+1}-\lambda^*, A(u_{k+1}-x^*)+B(v_{k+1}-y^*)\rangle-\frac{1}{2\gamma}\|\lambda_{k+1}-\lambda_k\|^2.\nonumber
\end{eqnarray}	
Adding $I_{k+1}^1-I_k^1$ to $I_{k+1}^4-I_k^4$, we obtain
\begin{eqnarray}\label{eq_le1_11}
	&&\mathcal{E}_{k+1}-\mathcal{E}_k = I_{k+1}^1-I_k^1+I_{k+1}^2-I_k^2+I_{k+1}^3-I_k^3+I_{k+1}^4-I_k^4\nonumber \\
	&&\leq \frac{t^2_{k+2}-t^2_{k+1}-\beta\mu_g t_{k+1}}{2\beta}\|v_{k+1}-y^*\|^2-\frac{1-\alpha L_{f_2}}{2\alpha}\|u_{k+1}-u_k\|^2\\
	&&\quad-\frac{t^2_{k+1}-\beta L_{g_2}}{2\beta}\|v_{k+1}-v_k\|^2-\frac{1}{2\gamma}\|\lambda_{k+1}-\lambda_k\|^2-\frac{\mu_g t_{k+1}^2(t_{k+1}-1)}{2}\|y_{k+1}-y_k\|^2\nonumber\\
	&&\quad+t_{k+1}(\langle A(u_{k+1}-x^*),\lambda_{k+1}-\bar{\lambda}_{k+1}\rangle+\langle B(v_{k+1}-y^*),\lambda_{k+1}- \hat{\lambda}_{k+1}\rangle). \nonumber
\end{eqnarray}
It follows from \eqref{disB5}  and $Ax^*+By^*-b=0$ that 
\begin{eqnarray}\label{eq_le1_12}
	&& t_{k+1}(\langle A(u_{k+1}-x^*),\lambda_{k+1}-\bar{\lambda}_{k+1}\rangle+\langle B(v_{k+1}-y^*),\lambda_{k+1}- \hat{\lambda}_{k+1}\rangle)\nonumber\\
	&& \qquad = \frac{1}{\gamma}\langle \lambda_{k+1}-\lambda_k, \lambda_{k+1}-\bar{\lambda}_{k+1}\rangle+t_{k+1}\langle B(v_{k+1}-y^*),\bar{\lambda}_{k+1}- \hat{\lambda}_{k+1}\rangle\\
	&& \qquad \leq \frac{1}{2\gamma}(\|\lambda_{k+1}-\lambda_k\|^2+\|\lambda_{k+1}-\bar{\lambda}_{k+1}\|^2)+t_{k+1}\langle B(v_{k+1}-y^*),\bar{\lambda}_{k+1}- \hat{\lambda}_{k+1}\rangle,\nonumber
\end{eqnarray}
which together \eqref{eq_le1_11}  implies the desired  result.
\end{proof}

\subsection{Convergence analysis}
With Lemma \ref{le_A1} in hand,  let's now delve into  the convergence properties of Algorithm \ref{al_2A} and Algorithm \ref{al_2B}.

%Given that Algorithm \ref{al_2A} corresponds to scheme \eqref{disB} with 	$\bar{\lambda}_{k+1} = \lambda_k+ \gamma t_{k+1}(Au_{k+1}+Bv_{k}-b)$ and $\hat{\lambda}_{k+1} =  {\lambda}_{k+1}$, while Algorithm \ref{al_2B} corresponds to scheme \eqref{disB} with  $\bar{\lambda}_{k+1}=\hat{\lambda}_{k+1} = \lambda_k+ \gamma t_{k+1}(Au_{k+1}+Bv_{k}-b)$, we can apply Lemma \ref{le_B1} and similar arguments as those in Theorem \ref{th_A1} and Theorem \ref{th_A2}. This leads to the following convergence results for Algorithm \ref{al_2A} and Algorithm \ref{al_2B}.

\begin{theorem}\label{th_B1}

Suppose that Assumption \ref{ass}  holds. Let $\{(x_k,y_k)\}_{k\geq 1}$ be the sequence generated by Algorithm \ref{al_2A} and  $(x^*,y^*,\lambda^*)\in\Omega$.  Assume that $\alpha\leq \frac{1}{L_{f_2}}$, $\beta\leq \frac{t_1^2}{L_{g_2}}$, and $\{t_k\}_{k\geq 1}\subseteq [1,+\infty)$ is a nondecreasing sequence satisfying
\begin{eqnarray*}
	t_{k+1}^2\leq t^2_k + t_{k+1}\quad\text{ and }\quad t^2_{k+1}\leq t^2_{k}+\frac{\beta\mu_g}{1+\beta\gamma\|B\|^2} t_{k}.
\end{eqnarray*}
Then the following conclusions hold:
\begin{eqnarray*}
&& \mathcal{L}(x_k,y_k,\lambda^*)-\mathcal{L}(x^*,y^*,\lambda^*)\leq \frac{\mathcal{E}_{1}+\frac{\gamma t^2_{2}}{2}\|B(y_{1}-y^*)\|^2}{t^2_k},\\
&&  \|Ax_{k}+By_{k}-b\|\leq  \frac{C_1}{t_{k}^2},\\
&& |f(x_k)+g(y_k)-f(x^*)-g(y^*)|\leq \frac{\mathcal{E}_{1}+\frac{\gamma t^2_{2}}{2}\|B(y_{1}-y^*)\|^2+C_1\|\lambda^*\|}{t_{k}^2},
\end{eqnarray*}
where 
\begin{equation}\label{E_1}
	\mathcal{E}_{1} = t^2_1(\mathcal{L}(x_1,y_1,\lambda^*)-\mathcal{L}(x^*,y^*,\lambda^*))+\frac{1}{2\alpha}\|x_1-x^*\|^2+\frac{t^2_{1}}{2\beta}\|y_1-y^*\|^2+\frac{1}{2\gamma}\|\lambda_1-\lambda^*\|^2
\end{equation}
 and 
 \begin{equation}\label{C_1}
 C_1= 3t^2_1\|Ax_1+By_1-b\|+\frac{2\|\lambda_1-\lambda^*\|+2\sqrt{2\gamma\mathcal{E}_{1}+{\gamma^2 t^2_{2}}\|B(y_{1}-y^*)\|^2}}{\gamma}. 	
 \end{equation}\end{theorem}
\begin{proof}
Since Algorithm \ref{al_2A} is equivalent to scheme \eqref{disB} with 	
	\[ \bar{\lambda}_{k+1} = \lambda_k+ \gamma t_{k+1}(Au_{k+1}+Bv_{k}-b), \quad\hat{\lambda}_{k+1} =  {\lambda}_{k+1},\]
we can establish the following relations
\[{\lambda}_{k+1} - \bar{\lambda}_{k+1} = \hat{{\lambda}}_{k+1}  - \bar{\lambda}_{k+1} = \gamma  t_{k+1}B(v_{k+1}-v_k).\]
This, combined with \eqref{eq_know}, leads to the following result
\begin{eqnarray*}
&&	\frac{1}{2\gamma}\|\lambda_{k+1}-\bar{\lambda}_{k+1}\|^2+t_{k+1}\langle B(v_{k+1}-y^*),\bar{\lambda}_{k+1}- \hat{\lambda}_{k+1}\rangle\\
	&&\qquad = \frac{\gamma t^2_{k+1}}{2}\|B(v_{k+1}-v_k)\|^2-\gamma t^2_{k+1}\langle B(v_{k+1}-y^*),B(v_{k+1}-v_k)\rangle\\
	&&\qquad=-\frac{\gamma t^2_{k+1}}{2}(\|B(v_{k+1}-y^*)\|^2 -\|B(v_{k}-y^*)\|^2)\\
	&&\qquad\leq -\left( \frac{\gamma t^2_{k+2}}{2}\|B(v_{k+1}-y^*)\|^2 -\frac{\gamma t^2_{k+1}}{2}\|B(v_{k}-y^*)\|^2\right)\\
	&&\qquad\quad+\frac{\gamma\|B\|^2 ( t^2_{k+2}-t^2_{k+1})}{2}\|v_{k+1}-y^*\|^2,
\end{eqnarray*}
where the last inequality is derived from the nondecreasing property of $t_k$. This, in conjunction with Lemma \ref{le_B1} and parameters setting, yields
\begin{eqnarray*}
	&& \mathcal{E}_{k+1} + \frac{\gamma t^2_{k+2}}{2}\|B(v_{k+1}-y^*)\|^2 - \left(\mathcal{E}_{k} + \frac{\gamma t^2_{k+1}}{2}\|B(v_{k}-y^*)\|^2\right)\\
	&&\qquad \leq \frac{(1+\beta\gamma\|B\|^2)(t^2_{k+2}-t^2_{k+1})-\beta\mu_g t_{k+1}}{2\beta}\|v_{k+1}-y^*\|^2\\
	&&\qquad\quad-\frac{1-\alpha L_{f_2}}{2\alpha}\|u_{k+1}-u_k\|^2-\frac{t^2_{k+1}-\beta L_{g_2}}{2\beta}\|v_{k+1}-v_k\|^2\\
	&&\qquad \leq 0,
\end{eqnarray*}
This implies
\[ \mathcal{E}_{k} + \frac{\gamma t^2_{k+1}}{2}\|B(v_{k}-y^*)\|^2 \leq \mathcal{E}_{k-1} + \frac{\gamma t^2_{k}}{2}\|B(v_{k-1}-y^*)\|^2\leq \mathcal{E}_{1} + \frac{\gamma t^2_{2}}{2}\|B(v_{1}-y^*)\|^2. \]
Then, using the definition of  $\mathcal{E}_{k}$, we obtain
\begin{eqnarray}\label{eq_th1_1}
	\mathcal{L}(x_k,y_k,\lambda^*)-\mathcal{L}(x^*,y^*,\lambda^*)= \frac{I_k^1}{t^2_k}\leq \frac{\mathcal{E}_{k}}{t^2_k}\leq \frac{\mathcal{E}_{1}+\frac{\gamma t^2_{2}}{2}\|B(v_{1}-y^*)\|^2}{t^2_k}
\end{eqnarray}
and 
\begin{eqnarray}\label{eq_th1_2}
	\|\lambda_k-\lambda^*\| = \sqrt{2\gamma I_k^4}\leq  \sqrt{2\gamma\mathcal{E}_{k}}\leq  \sqrt{2\gamma\mathcal{E}_{1}+{\gamma^2 t^2_{2}}\|B(v_{1}-y^*)\|^2} .
\end{eqnarray}
For the sake of simplicity in notation, let's define two sequences  $\{h_k\}_{k\geq 1}$ and $\{a_k\}_{k\geq 1}$ as
\begin{eqnarray}\label{eq_th1_3}
	h_k:=  t^2_{k}(Ax_{k}+By_{k}-b),\qquad a_k := \frac{t_{k}^2-t_{k+1}(t_{k+1}-1)}{t^2_{k}}.
\end{eqnarray}
From Algorithm \ref{al_2A}, we know 
\begin{eqnarray*}\label{eq_th1_4}
	\quad u_{k+1} = t_{k+1}x_{k+1} - (t_{k+1}-1)x_{k},\quad v_{k+1} = t_{k+1}y_{k+1} - (t_{k+1}-1)y_{k},
\end{eqnarray*}
which implies that
\begin{eqnarray*}\label{eq_lambda_va}
	\lambda_{k+1} - \lambda_1 &=& \sum^{k}_{i=1}(\lambda_{i+1}-\lambda_{i})\nonumber\\
	&=& \sum^k_{i=1}\gamma t_{i+1}(Au_{i+1}+Bv_{i+1}-b)\nonumber\\
	&=& \gamma \sum^k_{i=1} \left[ t^2_{i+1}(Ax_{i+1}+By_{i+1}-b)-  t_{i+1}(t_{i+1}-1)(Ax_{i}+By_{i}-b)\right]\\
	&=& \gamma \sum^k_{i=1}\left[  h_{i+1}-h_{i}+ \frac{t_{i}^2-t_{i+1}(t_{i+1}-1)}{t^2_{i}}h_{i}\right]\nonumber\\
	&=& \gamma(h_{k+1}-h_{1})+ \gamma \sum^k_{i=1}a_{i}h_{i}.\nonumber
\end{eqnarray*}
This, combined with \eqref{eq_th1_2} and the triangle inequality, results in:
\begin{eqnarray*}
	\left\|h_{k+1} +\sum^k_{i=1}a_{i} h_{i}\right \|&\leq &\|h_1\|+\frac{\|\lambda_{k+1}-\lambda_1\|}{\gamma}\\
	&\leq&\|h_1\|+\frac{\|\lambda_1-\lambda^*\|+\sqrt{2\gamma\mathcal{E}_{1}+{\gamma^2 t^2_{2}}\|B(v_{1}-y^*)\|^2}}{\gamma}.
\end{eqnarray*}
Given our assumption, we know $a_k\in[0,1)$ for any $k\geq 1$. Applying Lemma \ref{le_app1} in Appendix, we find 
 \[\|h_k\| \leq 3\|h_1\| + \frac{2\|\lambda_1-\lambda^*\|+2\sqrt{2\gamma\mathcal{E}_{1}+{\gamma^2 t^2_{2}}\|B(v_{1}-y^*)\|^2}}{\gamma},\qquad \forall k\geq 1.\]
Now, from \eqref{eq_th1_3}, we deduce that
\[\|Ax_{k}+By_{k}-b\|=\frac{\|h_k\|}{t^2_{k}}\leq  \frac{3\|h_1\|}{t_{k}^2} + \frac{2\|\lambda_1-\lambda^*\|+2\sqrt{2\gamma\mathcal{E}_{1}+{\gamma^2 t^2_{2}}\|B(v_{1}-y^*)\|^2}}{\gamma t_{k}^2} \]
for any  $k\geq 1$.
This together with \eqref{lagran} and \eqref{eq_th1_1} implies
\begin{eqnarray*}
&&	|f(x_k)+g(y_k)-f(x^*)-g(y^*)|\leq  \mathcal{L}(x_k,y_k,\lambda^*)-\mathcal{L}(x^*,y^*,\lambda^*)+\|\lambda^*\|\|Ax_{k}+By_{k}-b\|\\
	&&\leq  \frac{1}{t_{k}^2}\left(\mathcal{E}_{1}+\frac{\gamma t^2_{2}}{2}\|B(v_{1}-y^*)\|^2\right)\\
	&&\quad+  \frac{\|\lambda^*\|}{t_{k}^2}\left(3\|h_1\|+\frac{2\|\lambda_1-\lambda^*\|+2\sqrt{2\gamma\mathcal{E}_{1}+{\gamma^2 t^2_{2}}\|B(v_{1}-y^*)\|^2}}{\gamma}\right).
\end{eqnarray*}
This concludes our proof.
\end{proof}

\begin{theorem}\label{th_B2}
Suppose that Assumption \ref{ass} holds. Let $\{(x_k,y_k)\}_{k\geq 1}$ be the sequence generated by Algorithm \ref{al_2B} and  $(x^*,y^*,\lambda^*)\in\Omega$.  Assume that $\alpha\leq \frac{1}{L_{f_2}}$, $\beta\leq \frac{t_1^2}{L_{g_2}+t_1^2\gamma\|B\|^2}$, and   $\{t_k\}_{k\geq 1}\subseteq [1,+\infty)$ is a nondecreasing sequence satisfying
\[ t_{k+1}^2\leq t^2_k + t_{k+1}\quad\text{ and }\quad t^2_{k+1}\leq t^2_{k}+\beta\mu_g t_{k}.\] 
 the following conclusions hold:
\begin{eqnarray*}
&& \mathcal{L}(x_k,y_k,\lambda^*)-\mathcal{L}(x^*,y^*,\lambda^*)\leq \frac{\mathcal{E}_{1}}{t^2_k},\\
&&  \|Ax_{k}+By_{k}-b\|\leq  \frac{C_2}{t_{k}^2},\\
&& |f(x_k)+g(y_k)-f(x^*)-g(y^*)|\leq \frac{\mathcal{E}_{1}+C_2\|\lambda^*\|}{t_{k}^2},
\end{eqnarray*}
where $\mathcal{E}_{1}$ is defined in \eqref{E_1} and
\begin{equation}\label{C_2}
	C_2= 3t^2_1\|Ax_1+By_1-b\|+\frac{2\|\lambda_1-\lambda^*\|+2\sqrt{2\gamma\mathcal{E}_{1}}}{\gamma}.
\end{equation} 
\end{theorem}

\begin{proof}
Since Algorithm \ref{al_2B} is equivalent to scheme \eqref{disB} with 	
	\[ \bar{\lambda}_{k+1} = \hat{\lambda}_{k+1} =\lambda_k+\gamma t_{k+1}(Au_{k+1}+Bv_{k}-b),\]
we get
\[\|{\lambda}_{k+1} - \bar{\lambda}_{k+1} \|^2= \|\gamma  t_{k+1}B(v_{k+1}-v_k)\|\leq \gamma^2t_{k+1}^2\|B\|^2\|v_{k+1}-v_k\|^2.\]
Since $\beta\leq \frac{t_1^2}{L_{g_2}+t_1^2\gamma\|B\|^2}$ and $\{t_k\}_{k\geq 1}$ is a nondecreasing, we have $\gamma\beta\|B\|^2<1$ and 
$$(1-\gamma\beta\|B\|^2) t^2_{k}-\beta L_{g_2}\geq (1-\gamma\beta\|B\|^2) t^2_{1}-\beta L_{g_2}\geq 0.$$
By Lemma \ref{le_B1} and assumption, we have
\begin{eqnarray*}
	\mathcal{E}_{k+1}-\mathcal{E}_k &\leq& \frac{t^2_{k+2}-t^2_{k+1}-\beta\mu_g t_{k+1}}{2\beta}\|v_{k+1}-y^*\|^2-\frac{1-\alpha L_{f_2}}{2\alpha}\|u_{k+1}-u_k\|^2\\
	&&-\frac{(1-\gamma\beta\|B\|^2) t^2_{k+1}-\beta L_{g_2}}{2\beta}\|v_{k+1}-v_k\|^2\\
	&\leq & 0.
\end{eqnarray*}
It yields 
\[ \mathcal{E}_{k}\leq \mathcal{E}_{k-1}\leq \mathcal{E}_{1},\qquad \forall k\geq 1.\]
By essentially same argument as in the proof of Theorem \ref{th_B1}, we can get the desired results.
\end{proof}

\begin{remark}
As shown in the proof of Lemma \ref{le_B1}, we can see that when  $f_2$  vanishes in problem \eqref{ques}, i.e.,  $f=f_1$, the parameter $\alpha$ in both Algorithm \ref{al_2A} and Algorithm  \ref{al_2B} can be chosen within the interval $(0,+\infty)$. Similarly,  when  $g_2$  vanishes in problem \eqref{ques},  the assumption $\beta\leq \frac{t_1^2}{L_{g_2}}$ in  Algorithm \ref{al_2A} can be omitted and the assumption  $\beta\leq \frac{t_1^2}{L_{g_2}+t_1^2\gamma\|B\|^2}$ in Algorithm \ref{al_2B}  can be reduced to $\beta\leq \frac{1}{\gamma\|B\|^2}$.  
	\end{remark}

\subsection{Reduction to accelerated linearized ALM based on Nesterov's first scheme}\label{sec24}

From Algorithm \ref{al_2A}, we can easily get an accelerated linearized ALM for solving the one-block linearly constrained problem \eqref{ques_one} in  the convex case,  and from  Algorithm \ref{al_2B} we can get another accelerated linearized ALM for solving problem \eqref{ques_one} in the strongly convex case.

{\bf Convex case.} Let $B=0$. Problem \eqref{ques} becomes

\[
				\min_{x\in  \mathbb{R}^{m},y\in  \mathbb{R}^{n}} \quad f(x)+g(y) \qquad s.t.  \  Ax=b
\]
with  $f(x) = f_1(x)+f_2(x)$ and  $g(y)=g_1(y)+g_2(y)$. The above problem is equivalent to the following two independent problems: The one-block linearly constrained problem with variable $x$
\begin{equation}\label{one_bolck}
	\min_{x\in  \mathbb{R}^{m}}\quad f(x) = f_1(x)+f_2(x) \qquad s.t.  \  Ax=b,
\end{equation}
and  the unconstrained problem  with variable $y$
\[
				\min_{y\in  \mathbb{R}^{n}} g(y)=g_1(y)+g_2(y).
\]
In this case, the results of Theorem \ref{th_B1} hold for any strongly convex function $g_1$ and  smooth convex function $g_2$. By examining Algorithm \ref{al_2A}, it is evident that the subproblems update for functions $f$ and $g$ are independent. When focusing only on the convex function $f$, we derive the accelerated linearized ALM (Algorithm \ref{al_2B1}) for solving the one-block linearly constrained problem \eqref{one_bolck}.
\begin{algorithm}
        \caption{Accelerated Linearized ALM-Nesterov's first Scheme (I) for problem \eqref{one_bolck}}
        \label{al_2B1}
        {\bf Initialize:} Let $x_1=x_0\in dom(f), \lambda_1\in\mathbb{R}^p$, $\alpha>0,\gamma>0$, $t_1\geq 1$. \\
        \For{$k = 1, 2,\cdots$}{
        $\bar{x}_k = x_k + \frac{t_k-1}{t_{k+1}}(x_k-x_{k-1})$.\\
      $x_{k+1} = \mathop{\arg\min}_x\left(f_1(x)+\langle A^T\lambda_k+\nabla f_2(\bar{x}_k),x\rangle+\frac{\gamma t^2_{k+1}}{2}\|A(x-x_k)+\frac{1}{t_{k+1}}(Ax_k-b)\|^2\right.$\\
     $\left.\qquad\qquad\qquad \qquad+\frac{1}{2\alpha}\left\|x-\bar{x}_k\right\|^2\right)$.\\
      $u_{k+1}=x_{k+1}+(t_{k+1}-1)(x_{k+1}-x_k)$\\
		$\lambda_{k+1} = \lambda_k + \gamma t_{k+1}(Au_{k+1}-b)$.
}
\end{algorithm}

Construct the energy sequence $\{\mathcal{E}_k\}_{k\geq 1}$ of Algorithm \ref{al_2B1} as
\begin{eqnarray*}
	\mathcal{E}_k =t^2_k(f(x_k)-f(x^*)+\langle \lambda^*,Ax_k-b\rangle)+\frac{1}{2\alpha}\|u_k-x^*\|^2+\frac{1}{2\gamma}\|\lambda_k-\lambda^*\|^2.
\end{eqnarray*}
Under the assumptions that $f_1$ is  closed and convex, $f_2$ is convex and  has a Lipschitz continuous gradient with constant $L_{f_2}$, and $t_{k+1}^2\leq t^2_k + t_{k+1}$, by essentially same argument as in the proof of Lemma \ref{le_B1} and Theorem \ref{th_B1} we can obtain that
\[\mathcal{E}_{k+1}\leq \mathcal{E}_k\leq \mathcal{E}_1,\quad k\geq 1.\]
Then we get the follow convergence rates of Algorithm \ref{al_2B1}.
\begin{theorem}\label{th_B11}
Suppose that $f_1$ is a proper, closed and convex function, $f_2$ is convex and  has a Lipschitz continuous gradient with constant $L_{f_2}$, and the saddle points set $\Omega$ is nonempty. Let $\{x_k\}_{k \geq 1}$ be the sequence generated by Algorithm \ref{al_2B1} and  $(x^*,\lambda^*)\in\Omega$.  Assume that  $\alpha\leq \frac{1}{L_{f_2}}$  and $\{t_k\}_{k\geq 1}\subseteq [1,+\infty)$ is a nondecreasing sequence satisfying
\begin{eqnarray*}
	t_{k+1}^2\leq t^2_k + t_{k+1} 
\end{eqnarray*}
Then the following conclusions hold:
\begin{eqnarray*}
&&  \|Ax_{k}-b\|\leq  \frac{C_3}{t_{k}^2},\\
&& |f(x_k)-f(x^*)|\leq \frac{\mathcal{E}_{1}+C_3\|\lambda^*\|}{t_{k}^2},
\end{eqnarray*}
where $\mathcal{E}_{1} = t^2_1(f(x_1)-f(x^*)+\langle \lambda^*,Ax_1-b\rangle)+\frac{1}{2\alpha}\|x_1-x^*\|^2+\frac{1}{2\gamma}\|\lambda_1-\lambda^*\|^2$
 and 
$
 C_3= 3t^2_1\|Ax_1-b\|+\frac{2\|\lambda_1-\lambda^*\|+2\sqrt{2\gamma\mathcal{E}_{1}}}{\gamma}. 	
$\end{theorem}
\begin{remark}
	When $A=0$, $b=0$, Algorithm \ref{al_2B1} reduces to Nesterov's fisrt scheme \eqref{FISTA} for solving unconstrained composite convex problems. The convergence results in Theorem \ref{th_B11} are also consistent with the convergence results of Nesterov's fisrt scheme \eqref{FISTA} in \cite{BeckIma,Chambolle,Attouch2018}.
\end{remark}

{\bf Strongly convex case.} Let $A=0$. By focusing only on the part of strongly convex function $g$ in Algorithm \ref{al_2B}, we obtain the accelerated linearized ALM (Algorithm \ref{al_2B2}) for solving  the one-block linearly constrained problem:
\begin{equation}\label{one_bolck_two}
	\min_{y\in  \mathbb{R}^{n}}\quad g(y) = g_1(y)+g_2(y) \qquad s.t.  \  By=b.
\end{equation}

\begin{algorithm}
        \caption{Accelerated Linearized ALM-Nesterov's first Scheme (II) for problem \eqref{one_bolck_two}}
        \label{al_2B2}
        {\bf Initialize:} Let $y_1=y_0=v_1\in dom(g), \lambda_1\in\mathbb{R}^p$, $\beta>0,\gamma>0$, $t_1\geq 1$.  \\
        \For{$k = 1, 2,\cdots$}{
        $\bar{y}_k = y_k + \frac{t_k-1}{t_{k+1}}(y_k-y_{k-1})$.\\
      $\bar{\lambda}_{k+1} =\lambda_k+ \gamma t_{k+1}(Bv_{k}-b)$.\\
      Let $\eta_k = \frac{\beta}{t_{k+1}^2+\beta\mu_g(t_{k+1}-1)}$.\\
      $y_{k+1} ={\bf Prox}_{{\eta_k},g_1}\left(\bar{y}_k -\eta_k(\mu_g(t_{k+1}-1)(\bar{y}_k-y_k)+B^T\bar{\lambda}_{k+1}+\nabla g_2(\bar{y}_k))\right)$.\\
      $v_{k+1}=y_{k+1}+(t_{k+1}-1)(y_{k+1}-y_k)$.\\
		$\lambda_{k+1} = \lambda_k + \gamma t_{k+1}(Bv_{k+1}-b)$.
}
\end{algorithm}
Construct the energy sequence $\{\mathcal{E}_k\}_{k\geq 1}$ of Algorithm \ref{al_2B2} as
\begin{eqnarray*}
	\mathcal{E}_k =t^2_k(g(y_k)-g(y^*)+\langle \lambda^*,By_k-b\rangle)+\frac{t_{k+1}^2}{2\beta}\|v_k-y^*\|^2+\frac{1}{2\gamma}\|\lambda_k-\lambda^*\|^2.
\end{eqnarray*}
Under the assumptions that $g_1$ is a closed and $\mu_g$-strongly convex function, $g_2$ is convex has a Lipschitz continuous gradient with constant $L_{g_2}$ and $t_{k+1}^2\leq t^2_k + t_{k+1}$, $ t^2_{k+1}\leq t^2_{k}+\beta\mu_g t_{k}$, by essentially same argument as in the proof of Lemma \ref{le_A1} and Theorem \ref{th_A2} we can obtain that
\[\mathcal{E}_{k+1}\leq \mathcal{E}_k\leq \mathcal{E}_1,\quad k\geq 1,\]

Then we get the follow convergence rates of Algorithm \ref{al_2B2}.

\begin{theorem}\label{th_B22}
Suppose that $g_1$ is a proper, closed and $\mu_g$-strongly convex function, $g_2$ is convex has a Lipschitz continuous gradient with constant $L_{g_2}$, and the saddle points set $\Omega$ is nonempty. Let $\{y_k\}_{k \geq 1}$ be the sequence generated by Algorithm \ref{al_2B2} and  $(y^*,\lambda^*)\in\Omega$. Assume that  $\beta\leq \frac{t_1^2}{L_{g_2}+t_1^2\gamma\|B\|^2}$ and $\{t_k\}_{k\geq 1}\subseteq [1,+\infty)$ is a nondecreasing sequence satisfying
\begin{eqnarray*}
	t_{k+1}^2\leq t^2_k + t_{k+1}\quad\text{ and }\quad t^2_{k+1}\leq t^2_{k}+\beta\mu_g t_{k}.
\end{eqnarray*}
Then the following conclusions hold:
\begin{eqnarray*}
&&  \|By_{k}-b\|\leq  \frac{C_4}{t_{k}^2},\\
&& |g(y_k)-g(y^*)|\leq \frac{\mathcal{E}_{1}+C_4\|\lambda^*\|}{t_{k}^2},
\end{eqnarray*}
where $\mathcal{E}_{1} = t^2_1(g(y_1)-g(y^*)+\langle \lambda^*,By_1-b\rangle)+\frac{t_1^2}{2\beta}\|y_1-y^*\|^2+\frac{1}{2\gamma}\|\lambda_1-\lambda^*\|^2$
 and 
$
 C_4= 3t^2_1\|By_1-b\|+\frac{2\|\lambda_1-\lambda^*\|+2\sqrt{2\gamma\mathcal{E}_{1}}}{\gamma}. 	
$.
\end{theorem}

\begin{remark}
	Incorporating the appropriate parameter selection, we can obtain non-ergodic $\mathcal{O}(1/k^2)$ convergence rate results for both Algorithm \ref{al_2B1} and Algorithm \ref{al_2B2}. In the convex case, the $\mathcal{O}(1/k^2)$ convergence rate of Algorithm \ref{al_2B1} is consistent with the convergence results of the accelerated linearized ALM in \cite{HeNA}, which also employed Nesterov's first scheme \eqref{FISTA} within a primal-dual framework.	In the strongly convex case,  Algorithm \ref{al_2B2} allows to linearize the augmented term of Algorithm \ref{al_2B1}, the $\mathcal{O}(1/k^2)$ convergence rate is also consistent with the convergence results of  the accelerated linearized ALM in \cite{Sabach2022}.
	\end{remark}

\subsection{Parameter considerations}\label{sec_para}

In this part, we will provide a concise discussion on the selection scheme of parameters on Algorithm \ref{al_2A}. Similar discussions can be carried out for Algorithm \ref{al_2B} and other methods introduced in this paper. To avoid repetition, we will  focus on Algorithm \ref{al_2A} only.

 Let $a=\frac{\beta\mu_g}{1+\beta\gamma\|B\|^2}  $ and take
\begin{equation}\label{max_tk}
	 t_{k+1} = \min\left\{\frac{1+\sqrt{1+4t_k^2}}{2},\sqrt{t_k^2+at_k}\right\}.
\end{equation}
Under this setting, the assumptions of Theorem \ref{th_B1} are satisfied and the sequence ${t_k}$ grows rapidly. Since $t_1\geq 1$, following Lemma \ref{le_appt1} and Lemma \ref{le_appt2}, we have
\[t_k\geq t_1+\min\{1,2b\}\frac{k-1}{2}\geq \min\{1,2b\}\frac{k+1}{2}, \]
where $b=\frac{2at_1}{a+4t_1}< \frac{a}{2}$. As a result, we have the following non-ergodic $\mathcal{O}(1/k^2)$ convergence rates:
\begin{eqnarray*}
&& \mathcal{L}(x_k,y_k,\lambda^*)-\mathcal{L}(x^*,y^*,\lambda^*)\leq \frac{4}{\min \{1,4b^2\}}\frac{\mathcal{E}_{1}+\frac{\gamma t^2_{2}}{2}\|B(y_{1}-y^*)\|^2}{(k+1)^2},\\
&&  \|Ax_{k}+By_{k}-b\|\leq  \frac{4}{\min \{1,4b^2\}}\frac{C_1}{(k+1)^2},\\
&& |f(x_k)+g(y_k)-f(x^*)-g(y^*)|\leq\frac{4}{\min \{1,4b^2\}} \frac{\mathcal{E}_{1}+\frac{\gamma t^2_{2}}{2}\|B(y_{1}-y^*)\|^2+C_1\|\lambda^*\|}{(k+1)^2}.
\end{eqnarray*} 
If we choose parameters $\beta,\gamma, t_1$ as follows:
\begin{equation*}
	t_1> \sqrt{\frac{L_{g_2}}{\mu_g}},\quad \beta\in\left(\frac{1}{\mu_g},\frac{t_1^2}{L_{g_2}}\right],\quad \gamma< \frac{\beta\mu_g-1}{\beta\|B\|^2},
\end{equation*}
 which ensure $a>1$. Moreover, let $t_1$  also satisfy $t_1\geq \frac{a}{4(a-1)}$. This implies $b\geq \frac{1}{2}$, and consequently $t_{k+1}\geq \frac{k+1}{2}$. Thus, we can obtain the optimal convergence results for Algorithm \ref{al_2A}. 
 
If $t_k$ is nondecreasing and satisfies $t_{k+1}^2\leq t^2_k + t_{k+1}$, we have $(t_{k+1}-\frac{1}{2})^2\leq  t^2_k+\frac{1}{4}\leq (t_{k}+\frac{1}{2})^2$, which leads to  $t_{k+1}\leq t_k+1$. For any $a>1$, taking $t_1\geq \frac{1}{a-1}$ implies $\beta>\frac{1}{\mu_g}\left(1+\frac{1}{t_1}\right)$ and $\gamma< \frac{t_1(\beta\mu_g-1)-1}{(t_1+1)\beta\|B\|^2}$. As a result, we have $t_{k+1}-at_k\leq (1-a)t_{k}+1\leq(1-a)t_{1}+1\leq 0$. This, combined with  $t_{k+1}^2\leq t^2_k + t_{k+1}$, implies 
$t_{k+1}^2\leq t^2_k + at_{k}$.  By Theorem \ref{th_B1}, we can obtain the following results.

\begin{corollary}\label{cor_para}
Suppose that Assumption \ref{ass} holds. Let $\{(x_k,y_k)\}_{k\geq 1}$ be the sequence generated by Algorithm \ref{al_1A} and  $(x^*,y^*,\lambda^*)\in\Omega$.  Assume that 
\begin{equation}\label{para_optim}
		t_1> \max\left\{1,\sqrt{\frac{2L_{g_2}}{\mu_g}}\right\},\ \  \alpha\leq \frac{1}{L_{f_2}},\ \  \beta\in\left(\frac{1}{\mu_g}\left(1+\frac{1}{t_1}\right),\frac{t_1^2}{L_{g_2}}\right],\ \  \gamma< \frac{t_1(\beta\mu_g-1)-1}{(t_1+1)\beta\|B\|^2},  
\end{equation}
and $\{t_k\}_{k\geq 1}$ is a nondecreasing sequence satisfying
\begin{eqnarray}\label{eq_cor_tk}
	t_{k+1}^2\leq t^2_k + t_{k+1}.
\end{eqnarray}
Then the following conclusions hold:
\begin{eqnarray*}
&& \mathcal{L}(x_k,y_k,\lambda^*)-\mathcal{L}(x^*,y^*,\lambda^*)\leq \frac{\mathcal{E}_{1}+\frac{\gamma t^2_{2}}{2}\|B(y_{1}-y^*)\|^2}{t^2_k},\\
&&  \|Ax_{k}+By_{k}-b\|\leq  \frac{C_1}{t_{k}^2},\\
&& |f(x_k)+g(y_k)-f(x^*)-g(y^*)|\leq \frac{\mathcal{E}_{1}+\frac{\gamma t^2_{2}}{2}\|B(y_{1}-y^*)\|^2+C_1\|\lambda^*\|}{t_{k}^2}.
\end{eqnarray*}
\end{corollary}

\begin{remark}
When \eqref{eq_cor_tk} holds with equality, i.e.,
\[ t_{k+1}=\frac{1+\sqrt{1+4t_k^2}}{2},\]
it becomes the classical Nesterov extrapolation  in \cite{Nesterov1983}.
 When  $f_2$  vanishes in problem \eqref{ques}, i.e.,  $f=f_1$, the parameter assumption $\alpha\leq\frac{1}{L_{f_2}}$ in \eqref{para_optim} can be omitted. When $g_2$  vanishes in problem \eqref{ques},  we can take $t_1\geq 1$ and $\beta\geq \frac{1}{\mu_g}\left(1+\frac{1}{t_1}\right)$  in \eqref{para_optim}.
	\end{remark}

\section{Accelerated Linearized ADMM Based on  Nesterov's Second Scheme}\label{sec_2}
\subsection{Formulation of algorithm}
In this part, we extend Nesterov's second scheme \eqref{NestSS} to the ADMM framework. By using the optimality condition, algorithm  \eqref{NestSS} is equivalent to the following scheme: 
\begin{numcases}{}
        {y}_k = x_k + \frac{t_k-1}{t_{k+1}}(x_k-x_{k-1}),\nonumber \\
		u_{k+1}-u_k = -{s t_{k+1}}(\widetilde{\nabla} g(\boldsymbol{u_{k+1}})+\nabla f(y_k)),\label{dis_Nest1}\\
		x_{k+1} = \frac{1}{t_{k+1}}u_{k+1}+ \frac{t_{k+1}-1}{t_{k+1}}x_{k}.\nonumber
 \end{numcases}
Motivated by \eqref{dis_Nest1}, we associate  problem \eqref{ques} to the following scheme:
\begin{subequations}\label{disA}
\begin{numcases}{}
        (\bar{x}_k, \bar{y}_k) = (x_k,y_k) + \frac{t_k-1}{t_{k+1}}[(x_k,y_k)-(x_{k-1},y_{k-1})],\label{disA1}\\
		u_{k+1}-u_k = -\boldsymbol{\alpha t_{k+1}}(\widetilde{\nabla} f_1(\boldsymbol{u_{k+1}})+\nabla f_2(\bar{x}_k)+\boldsymbol{A^T\bar{\lambda}_{k+1}}),\label{disA2}\\
		v_{k+1}-v_k = -\boldsymbol{\frac{\beta}{t_{k+1}}}(\widetilde{\nabla} g_1(\boldsymbol{v_{k+1}})+\nabla g_2(\bar{y}_k) +\boldsymbol{B^T{\hat{\lambda}}_{k+1}}),\label{disA3}\\
		(x_{k+1},y_{k+1}) = \frac{1}{t_{k+1}}(u_{k+1},v_{k+1})+ \frac{t_{k+1}-1}{t_{k+1}}(x_{k},y_k),\label{disA4}\\
		\lambda_{k+1} = \lambda_k + \gamma t_{k+1}(Au_{k+1}+Bv_{k+1}-b),\label{disA5}
 \end{numcases}
\end{subequations}
where $x_1=x_0\in\mathbb{R}^m, y_1=y_0\in\mathbb{R}^n$, $u_1 = x_1$, $v_1=y_1$, $\alpha>0$, $\beta>0$, $\gamma>0$, and $t_{k}\geq 1$ is a nondecreasing sequence. We also choose $\bar{\lambda}_{k+1}$  and  $\hat{\lambda}_{k+1}$ as
\begin{eqnarray*}\label{lambda_choice}
	 \bar{\lambda}_{k+1} = \lambda_k +\gamma t_{k+1}(Au_{k+1}+Bv_{k}-b),\qquad \hat{\lambda}_{k+1} = {\lambda}_{k+1} \text{ or } \bar{\lambda}_{k+1} .
\end{eqnarray*}
It can be observed from \eqref{disA2} and  \eqref{disA4} that the update of the primal variable $x_k$ is same as  \eqref{dis_Nest1}.

Since $\bar{\lambda}_{k+1} = \lambda_k +\ \gamma t_{k+1}(Au_{k+1}+Bv_{k}-b)$,  using the optimality condition of \eqref{disA2}, we have
\begin{eqnarray}\label{updata_uA1}
	u_{k+1} &=& \mathop{\arg\min}_u \left( f_1(u)+\langle A^T\lambda_k+\nabla f_2(\bar{x}_k),u\rangle+\frac{\gamma t_{k+1}}{2}\|Au+Bv_k-b\|^2\right.\nonumber \\
	&& \qquad\qquad\left.+\frac{1}{2\alpha t_{k+1}}\left\|u-u_k\right\|^2\right).
\end{eqnarray}
Similarly, when $\hat{\lambda}_{k+1} ={\lambda}_{k+1}$,  equation \eqref{disA3} is equivalent to
\begin{eqnarray}\label{updata_vA1}
	v_{k+1} &=&\mathop{\arg\min}_v \left(g_1(v)+\langle B^T\lambda_k+\nabla g_2(\bar{y}_k),v\rangle+\frac{\gamma t_{k+1}}{2}\|Au_{k+1}+Bv-b\|^2\right. \nonumber\\
	&&\left.\qquad\qquad+\frac{t_{k+1}}{2\beta}\left\|v-v_k\right\|^2.\right)
\end{eqnarray}
In this case, we can rewrite equation \eqref{disA} as Algorithm \ref{al_1A},  permitting linearization of the smooth functions $f_2$ and $g_2$.

When $\hat{\lambda}_{k+1} =\bar{\lambda}_{k+1}$,   equation \eqref{disA3} is equivalent to
\begin{equation}\label{updata_vA2}
	v_{k+1} ={\bf Prox}_{\frac{\beta}{t_{k+1}},g_1}\left(v_k-\frac{\beta}{t_{k+1}}(B^T\bar{\lambda}_{k+1}+\nabla g_2(\bar{y}_k)\right).\end{equation}
In this case, we can rewrite equation  \eqref{disA} as Algorithm \ref{al_1B},  permitting further  linearization of  the augmented term $\|Au_{k+1}+Bv-b\|^2$ in Algorithm \ref{al_1A}.

\begin{algorithm}
        \caption{Accelerated Linearized ADMM-Nesterov's Second Scheme (I)}
        \label{al_1A}
        {\bf Initialize:} Let $x_1=x_0=u_1\in dom(f), y_1=y_0=v_1\in dom(g), \lambda_1\in\mathbb{R}^p$, $\alpha>0,\beta>0,\gamma>0$, $t_1\geq 1$. \\
        \For{$k = 1, 2,\cdots$}{
        $(\bar{x}_k, \bar{y}_k) = (x_k,y_k) + \frac{t_k-1}{t_{k+1}}[(x_k,y_k)-(x_{k-1},y_{k-1})]$.\\
      $u_{k+1} = \mathop{\arg\min}_u\left(f_1(u)+\langle A^T\lambda_k+\nabla f_2(\bar{x}_k),u\rangle+\frac{\gamma t_{k+1}}{2}\|Au+Bv_k-b\|^2+\frac{1}{2\alpha t_{k+1}}\left\|u-u_k\right\|^2\right)$. \\
      $v_{k+1} =\mathop{\arg\min}_v\left(g_1(v)+\langle B^T\lambda_k+\nabla g_2(\bar{y}_k),v\rangle+\frac{\gamma t_{k+1}}{2}\|Au_{k+1}+Bv-b\|^2+\frac{t_{k+1}}{2\beta}\left\|v-v_k\right\|^2\right)$. \\
     $(x_{k+1},y_{k+1}) = \frac{1}{t_{k+1}}(u_{k+1},v_{k+1})+ \frac{t_{k+1}-1}{t_{k+1}}(x_{k},y_k)$.\\
		$\lambda_{k+1} = \lambda_k + \gamma t_{k+1}(Au_{k+1}+Bv_{k+1}-b)$.
}
\end{algorithm}

\begin{algorithm}
        \caption{Accelerated Linearized ADMM-Nesterov's Second Scheme (II)}
        \label{al_1B}
        {\bf Initialize:} Let $x_1=x_0=u_1\in dom(f), y_1=y_0=v_1\in dom(g), \lambda_1\in\mathbb{R}^p$, $\alpha>0,\beta>0,\gamma>0$, $t_1\geq 1$. \\
        \For{$k = 1, 2,\cdots$}{
        $(\bar{x}_k, \bar{y}_k) = (x_k,y_k) + \frac{t_k-1}{t_{k+1}}[(x_k,y_k)-(x_{k-1},y_{k-1})]$.\\
      $u_{k+1} = \mathop{\arg\min}_u \left( f_1(u)+\langle A^T\lambda_k+\nabla f_2(\bar{x}_k),u\rangle+\frac{\gamma t_{k+1}}{2}\|Au+Bv_k-b\|^2+\frac{1}{2\alpha t_{k+1}}\left\|u-u_k\right\|^2\right)$. \\
      $\bar{\lambda}_{k+1} =\lambda_k+ \gamma t_{k+1}(Au_{k+1}+Bv_{k}-b)$.\\
      $v_{k+1} ={\bf Prox}_{\frac{\beta}{t_{k+1}},g_1}\left(v_k-\frac{\beta}{t_{k+1}}(B^T\bar{\lambda}_{k+1}+\nabla g_2(\bar{y}_k)\right) $. \\
     $(x_{k+1},y_{k+1}) = \frac{1}{t_{k+1}}(u_{k+1},v_{k+1})+ \frac{t_{k+1}-1}{t_{k+1}}(x_{k},y_k)$.\\
		$\lambda_{k+1} = \lambda_k + \gamma t_{k+1}(Au_{k+1}+Bv_{k+1}-b)$.
}
\end{algorithm}

\subsection{A single-step analysis}
Next, we present a single-iteration analysis for the sequence that adheres to \eqref{disA}. These results will be utilized in the convergence analysis of both Algorithm \ref{al_1A} and Algorithm \ref{al_1B}. 

\begin{lemma}\label{le_A1}
Suppose that Assumption \ref{ass} holds. Let $\{(x_k,y_k,\lambda_k,u_k,v_k,\bar{\lambda}_k, \hat{\lambda}_k)\}_{k \geq 1}$ be the sequence satisfying \eqref{disA}, with $\alpha,\beta,\gamma>0$ and $\{t_k\}_{k\geq 1}\subseteq [1,+\infty)$ such that
\[ t_{k+1}^2\leq t^2_k + t_{k+1},\qquad k\geq 1.\]
 Then  for any $(x^*,y^*,\lambda^*)\in\Omega$, we have
\begin{eqnarray*} 
	&&\mathcal{E}_{k+1}-\mathcal{E}_k \leq \frac{t^2_{k+2}-t^2_{k+1}-\beta\mu_g t_{k+1}}{2\beta}\|v_{k+1}-y^*\|^2-\frac{1-\alpha L_{f_2}}{2\alpha}\|u_{k+1}-u_k\|^2\\
	&&-\frac{t^2_{k+1}-\beta L_{g_2}}{2\beta}\|v_{k+1}-v_k\|^2+\frac{1}{2\gamma}\|\lambda_{k+1}-\bar{\lambda}_{k+1}\|^2+t_{k+1}\langle B(v_{k+1}-y^*),\bar{\lambda}_{k+1}- \hat{\lambda}_{k+1}\rangle.
\end{eqnarray*}
\end{lemma}

\begin{proof}
Since $f_1$ and $g_1$ are convex functions, we can deduce from  \eqref{disA4} and $t_k\geq 1$ that
	\begin{eqnarray*}\label{eq_le1_1}
	&& f_1(x_{k+1})+g_1(y_{k+1})+\langle \lambda^*,Ax_{k+1}+By_{k+1} - b\rangle\nonumber\\
	&& \qquad\leq\frac{1}{t_{k+1}}(f_1(u_{k+1})+g_1(v_{k+1})+\langle \lambda^*,Au_{k+1}+Bv_{k+1} - b\rangle)\\
	&&\qquad\quad+\frac{t_{k+1}-1}{t_{k+1}}(f_1(x_{k})+g_1(y_{k})+\langle \lambda^*,Ax_{k}+By_{k} - b\rangle).\nonumber
	\end{eqnarray*}
Combining this inequality with $t_{k+1}^2\leq t^2_k + t_{k+1}$, we have
\begin{eqnarray}\label{eq_le1_2}
	I_{k+1}^1-I_k^1	&= &(t_{k+1}(t_{k+1}-1)-t^2_k)(\mathcal{L}(x_k,y_k,\lambda^*)-\mathcal{L}(x^*,y^*,\lambda^*))\nonumber\\
	&& + t^2_{k+1}\mathcal{L}(x_{k+1},y_{k+1},\lambda^*)-t_{k+1}(t_{k+1}-1)\mathcal{L}(x_k,y_k,\lambda^*)- t_{k+1}\mathcal{L}(x^*,y^*,\lambda^*)\\
	&\leq & t_{k+1}\left(f_1(u_{k+1})+g_1(v_{k+1})-f_1(x^*)-g_1(y^*)+\langle \lambda^*,Au_{k+1}+Bv_{k+1} - b\rangle\right)\nonumber\\
	&& +t_{k+1}(f_2(x_{k+1})+g_2(y_{k+1})-f_2(x^*)-g_2(y^*))\nonumber\\
	&&+t_{k+1}(t_{k+1}-1)(f_2(x_{k+1})+g_2(y_{k+1})-f_2(x_k)-g_2(y_k)).\nonumber
\end{eqnarray}
By exploiting the convexity of $f_1$, we have
\[\langle u_{k+1}-x^*, \widetilde{\nabla} f_1(u_{k+1})\rangle \geq f_1(u_{k+1})-f_1(x^*).\]
Then, by similar arguments as \eqref{eq_le2_2} and  $u_{k+1} = x_{k+1} +(t_{k+1}-1)(x_{k+1}-x_{k})$, we have 
\begin{eqnarray}\label{eq_heaaA}
	I_{k+1}^2-I_k^2 &=& -t_{k+1} \langle u_{k+1}-x^*, \widetilde{\nabla} f_1(u_{k+1})+\nabla f_2(\bar{x}_k)+A^T\lambda^*\rangle\nonumber\\
	&&-t_{k+1} \langle u_{k+1}-x^*,A^T(\bar{\lambda}_{k+1}-\lambda^*)\rangle-\frac{1}{2\alpha}\|u_{k+1}-u_k\|^2\\
	&\leq & -t_{k+1}(f_1(u_{k+1})-f_1(x^*)+\langle \lambda^*,A(u_{k+1}-x^*)\rangle)\nonumber\\
	&& -t_{k+1}(f_2(x_{k+1})-f_2(x^*))-t_{k+1}(t_{k+1}-1)(f_2(x_{k+1})-f_2(x_k)) \nonumber\\
	&& -t_{k+1} \langle u_{k+1}-x^*,A^T(\bar{\lambda}_{k+1}-\lambda^*)\rangle-\frac{1-\alpha L_{f_2}}{2\alpha}\|u_{k+1}-u_k\|^2, \nonumber
\end{eqnarray}	
Similarly, from \eqref{eq_le2_3}, we get
\begin{eqnarray}\label{eq_le1_8}
&&	I_{k+1}^3-I_k^3	 = \frac{t^2_{k+2}}{2\beta}\|v_{k+1}-y^*\|^2-\frac{t^2_{k+1}}{2\beta}\|v_{k}-y^*\|^2\nonumber\\
	&&\qquad\ \ =\frac{t^2_{k+2}-t^2_{k+1}}{2\beta}\|v_{k+1}-y^*\|^2+ \frac{t^2_{k+1}}{\beta}(\langle v_{k+1}-y^*, v_{k+1}-v_k\rangle-\frac{1}{2}\|v_{k+1}-v_k\|^2)\nonumber\\
&&\qquad\overset{\eqref{disA3}}{=} \frac{t^2_{k+2}-t^2_{k+1}}{2\beta}\|v_{k+1}-y^*\|^2-t_{k+1} \langle v_{k+1}-y^*, \widetilde{\nabla} g_1(v_{k+1}) + {\nabla} g_2(\bar{y}_{k})+B^T\lambda^*\rangle\nonumber\\
	&&\qquad\qquad-t_{k+1} \langle v_{k+1}-y^*,B^T(\hat{\lambda}_{k+1}-\lambda^*)\rangle-\frac{t^2_{k+1}}{2\beta}\|v_{k+1}-v_k\|^2\\
	&&\qquad\overset{\eqref{strongC}}{\leq}  \frac{t^2_{k+2}-t^2_{k+1}-\beta\mu_g t_{k+1}}{2\beta}\|v_{k+1}-y^*\|^2\nonumber\\
	&&\qquad\qquad -t_{k+1}(g_1(v_{k+1})-g_1(y^*)+\langle \lambda^*,B(v_{k+1}-y^*)\rangle)\nonumber\\
	&&\qquad\qquad- t_{k+1}(g_2(y_{k+1})-g_2(y^*))-t_{k+1}(t_{k+1}-1)(g_2(y_{k+1})-g_2(y_k))\nonumber\\
	&& \qquad\qquad-t_{k+1} \langle v_{k+1}-y^*,B^T(\hat{\lambda}_{k+1}-\lambda^*)\rangle-\frac{t^2_{k+1}-\beta L_{g_2}}{2\beta}\|v_{k+1}-v_k\|^2.\nonumber
\end{eqnarray}	

Combining \eqref{eq_le1_10} and \eqref{eq_le1_2} - \eqref{eq_le1_8}, we obtain
\begin{eqnarray*}
	\mathcal{E}_{k+1}-\mathcal{E}_k &=& I_{k+1}^1-I_k^1+I_{k+1}^2-I_k^2+I_{k+1}^3-I_k^3+I_{k+1}^4-I_k^4\nonumber\\
	&\leq& \frac{t^2_{k+2}-t^2_{k+1}-\beta\mu_g t_{k+1}}{2\beta}\|v_{k+1}-y^*\|^2-\frac{1-\alpha L_{f_2}}{2\alpha}\|u_{k+1}-u_k\|^2\nonumber\\
	&&-\frac{t^2_{k+1}-\beta L_{g_2}}{2\beta}\|v_{k+1}-v_k\|^2-\frac{1}{2\gamma}\|\lambda_{k+1}-\lambda_k\|^2\\
	&&+t_{k+1}(\langle A(u_{k+1}-x^*),\lambda_{k+1}-\bar{\lambda}_{k+1}\rangle+\langle B(v_{k+1}-y^*),\lambda_{k+1}- \hat{\lambda}_{k+1}\rangle).\nonumber
\end{eqnarray*}
This together \eqref{eq_le1_12} yields the desired result.
\end{proof}

\subsection{Convergence analysis}
Given that Algorithm \ref{al_1A} corresponds to scheme \eqref{disA} with 	$\bar{\lambda}_{k+1} = \lambda_k+ \gamma t_{k+1}(Au_{k+1}+Bv_{k}-b)$ and $\hat{\lambda}_{k+1} =  {\lambda}_{k+1}$, while Algorithm \ref{al_1B} corresponds to scheme \eqref{disA} with  $\bar{\lambda}_{k+1}=\hat{\lambda}_{k+1} = \lambda_k+ \gamma t_{k+1}(Au_{k+1}+Bv_{k}-b)$, we can apply Lemma \ref{le_A1} and similar arguments as those in Theorem \ref{th_B1} and Theorem \ref{th_B2}. This leads to the following convergence results for Algorithm \ref{al_1A} and Algorithm \ref{al_1B}.

\begin{theorem}\label{th_A1}
Suppose that Assumption \ref{ass} holds. Let $\{(x_k,y_k)\}_{k \geq 1}$ be the sequence generated by Algorithm \ref{al_1A} and  $(x^*,y^*,\lambda^*)\in\Omega$.  Assume that  $\alpha\leq \frac{1}{L_{f_2}}$, $\beta\leq \frac{t_1^2}{L_{g_2}}$, and $\{t_k\}_{k\geq 1}\subseteq [1,+\infty)$ is a nondecreasing sequence satisfying
\begin{eqnarray*}
	t_{k+1}^2\leq t^2_k + t_{k+1}  \quad\text{ and }\quad
	t^2_{k+1}\leq t^2_{k}+\frac{\beta\mu_g}{1+\beta\gamma\|B\|^2} t_{k}.
\end{eqnarray*}
Then the following conclusions hold:
\begin{eqnarray*}
&& \mathcal{L}(x_k,y_k,\lambda^*)-\mathcal{L}(x^*,y^*,\lambda^*)\leq \frac{\mathcal{E}_{1}+\frac{\gamma t^2_{2}}{2}\|B(y_{1}-y^*)\|^2}{t^2_k},\\
&&  \|Ax_{k}+By_{k}-b\|\leq  \frac{C_1}{t_{k}^2},\\
&& |f(x_k)+g(y_k)-f(x^*)-g(y^*)|\leq \frac{\mathcal{E}_{1}+\frac{\gamma t^2_{2}}{2}\|B(y_{1}-y^*)\|^2+C_1\|\lambda^*\|}{t_{k}^2},
\end{eqnarray*}
where $\mathcal{E}_{1}$ is defined in \eqref{E_1} and $C_1$ is defined in \eqref{C_1}.
\end{theorem}

\begin{theorem}\label{th_A2}
Suppose that Assumption \ref{ass} holds. Let $\{(x_k,y_k)\}_{k \geq 1}$ be the sequence generated by Algorithm \ref{al_1B} and  $(x^*,y^*,\lambda^*)\in\Omega$. Assume that $\alpha\leq \frac{1}{L_{f_2}}$, $\beta\leq \frac{t_1^2}{L_{g_2}+t_1^2\gamma\|B\|^2}$, and $\{t_k\}_{k\geq 1}\subseteq [1,+\infty)$ is a nondecreasing sequence satisfying
\begin{eqnarray*}
	t_{k+1}^2\leq t^2_k + t_{k+1}\quad\text{ and }\quad t^2_{k+1}\leq t^2_{k}+\beta\mu_g t_{k}.
\end{eqnarray*}
Then the following conclusions hold:
\begin{eqnarray*}
&& \mathcal{L}(x_k,y_k,\lambda^*)-\mathcal{L}(x^*,y^*,\lambda^*)\leq \frac{\mathcal{E}_{1}}{t^2_k},\\
&&  \|Ax_{k}+By_{k}-b\|\leq  \frac{C_2}{t_{k}^2},\\
&& |f(x_k)+g(y_k)-f(x^*)-g(y^*)|\leq \frac{\mathcal{E}_{1}+C_2\|\lambda^*\|}{t_{k}^2},
\end{eqnarray*}
where $\mathcal{E}_{1}$ is defined in \eqref{E_1} and $C_2$ is defined in \eqref{C_2}.
\end{theorem}

\begin{remark}
Note that the update of $v_{k+1}$ in Algorithm \ref{al_1B}  is  calculated with the help of the  proximal operator.  When slight adjustments are made to Algorithm \ref{al_1A}, the augmented term can also be linearized by the proximal technique. Specifically, the update of $v_{k+1}$ can be expressed as:
\begin{eqnarray*}
	v_{k+1} &=&\mathop{\arg\min}_v\left(g_1(v)+\langle B^T\lambda_k+\nabla g_2(\bar{y}_k),v\rangle+\frac{\gamma t_{k+1}}{2}\|Au_{k+1}+Bv-b\|^2\right.\\
	&&\left.\qquad\qquad+\frac{t_{k+1}}{2\beta}\left\|v-v_k\right\|^2_{\boldsymbol{Q}}\right),
\end{eqnarray*}
where $Q= \beta(\eta Id_n - \gamma B^TB)$ and $\eta\geq \gamma \|B\|^2$. Then 
\[ v_{k+1} ={\bf Prox}_{\frac{1}{\eta t_{k+1}},g_1}\left(\frac{1}{\eta}\left(\frac{1}{\beta}Qv_k-\gamma B^T(Au_{k+1}-b)-\frac{1}{t_{k+1}}(B^T\lambda_k+\nabla g_2(\bar{y}_k))\right)\right).\]
In this case, \eqref{disA2} changes as
\[\boldsymbol{Q}(v_{k+1}-v_k) = -\boldsymbol{\frac{\beta}{t_{k+1}}}(\widetilde{\nabla} g_1(\boldsymbol{v_{k+1}})+\nabla g_2(\bar{y}_k) +B^T\hat{\lambda}_{k+1}).\]
By replacing $I_3^k$ in \eqref{energy_sub} with $I_3^k=\frac{t^2_{k+1}}{2\beta}\|v_k-y^*\|_Q^2$, we can easily achieve the same convergence rate as in Theorem \ref{al_1A} under suitable parameter settings.
\end{remark}

\begin{remark}
By observing Algorithm \ref{al_2A}, Algorithm \ref{al_2B}, Algorithm \ref{al_1A} and Algorithm \ref{al_1B}, and the proofs of their convergence rate results, we can modify them to create new algorithms.
Specifically,  we can modify the updates of $(x_{k+1},u_{k+1})$ and $(y_{k+1},v_{k+1})$ in the proposed algorithms by the following two schemes: 

{\bf Scheme 1:} Update $x_{k+1}$ by \eqref{updata_uB1} and $u_{k+1}=x_{k+1}+(t_{k+1}-1)(x_{k+1}-x_k)$.  Update $v_{k+1}$ by \eqref{updata_vA1} or \eqref{updata_vA2} and 
     $y_{k+1} = \frac{1}{t_{k+1}}v_{k+1}+ \frac{t_{k+1}-1}{t_{k+1}}y_{k}$.

{\bf Scheme 2:} Update $u_{k+1}$ by \eqref{updata_uA1} and  $x_{k+1} = \frac{1}{t_{k+1}}u_{k+1}+ \frac{t_{k+1}-1}{t_{k+1}}x_{k}$.  Update $y_{k+1}$ by \eqref{updata_vB1} or  \eqref{updata_vB2} and 
      $v_{k+1}=y_{k+1}+(t_{k+1}-1)(y_{k+1}-y_k)$.
      
 The modified algorithms can achieve  non-ergodic convergence rates of  $O(1/t^2_k)$ through arguments similar to those in the proofs of   Lemma \ref{le_B1}, Lemma \ref{le_A1},  Theorem \ref{th_A1}, and Theorem \ref{th_A2}.
\end{remark}

\subsection{Reduction to accelerated linearized ALM based on Nesterov's second scheme}

Through a similar discussion as in Section \ref{sec24}, from Algorithm \ref{al_1A} and Algorithm \ref{al_1B}  we can get the following two accelerated linearized ALMs and the corresponding convergence  results.

\begin{algorithm}
        \caption{Accelerated Linearized ALM-Nesterov's Second Scheme (I) for Problem \eqref{one_bolck}}
        \label{al_1B1}
        {\bf Initialize:} Let $x_1=x_0\in\mathbb{R}^m, \lambda_1\in\mathbb{R}^p$, $\alpha>0,\gamma>0$, $t_1\geq 1$. \\
        \For{$k = 1, 2,\cdots$}{
        $\bar{x}_k = x_k + \frac{t_k-1}{t_{k+1}}(x_k-x_{k-1})$.\\
      $u_{k+1} = \mathop{\arg\min}_u\left(f_1(u)+\langle A^T\lambda_k+\nabla f_2(\bar{x}_k),u\rangle+\frac{\gamma t_{k+1}}{2}\|Au-b\|^2+\frac{1}{2\alpha t_{k+1}}\left\|u-u_k\right\|^2\right)$. \\
     $ x_{k+1} = \frac{1}{t_{k+1}} u_{k+1}+ \frac{t_{k+1}-1}{t_{k+1}}x_{k}$.\\
		$\lambda_{k+1} = \lambda_k + \gamma t_{k+1}(Au_{k+1}-b)$.
}
\end{algorithm}

\begin{theorem}\label{th_A11}
Suppose that $f_1$ is a closed and convex function, $f_2$ is convex and  has a Lipschitz continuous gradient with constant $L_{f_2}$, and the saddle points set $\Omega$ is nonempty. Let $\{x_k\}_{k \geq 1}$ be the sequence generated by Algorithm \ref{al_1B1} and  $(x^*,\lambda^*)\in\Omega$.  Assume that  $\alpha\leq \frac{1}{L_{f_2}}$ and $\{t_k\}_{k\geq 1}\subseteq [1,+\infty)$ is a nondecreasing sequence satisfying
\begin{eqnarray*}
	t_{k+1}^2\leq t^2_k + t_{k+1}.
\end{eqnarray*}
Then the following conclusions hold:
\begin{eqnarray*}
&&  \|Ax_{k}-b\|\leq  \frac{C_3}{t_{k}^2},\\
&& |f(x_k)-f(x^*)|\leq \frac{\mathcal{E}_{1}+C_3\|\lambda^*\|}{t_{k}^2},
\end{eqnarray*}
where $\mathcal{E}_{1}$ and $C_3$ are defined same as in Theorem \ref{th_B11}.

\end{theorem}

\begin{remark}
	When $A=0$ and $b=0$, Algorithm \ref{al_1B1} reduces to Nesterov's second scheme \eqref{NestSS} for solving unconstrained composite convex problems. The convergence results in Theorem \ref{th_A11} are also consistent with the convergence results of Nesterov's second scheme \eqref{NestSS} in \cite{Tseng2008,LiML2019ACC}.
\end{remark}

\begin{algorithm}
        \caption{Accelerated Linearized ALM-Nesterov's Second Scheme (II) for Problem \eqref{one_bolck_two}}
        \label{al_1B2}
        {\bf Initialize:} Let $y_1=y_0\in\mathbb{R}^n, \lambda_1\in\mathbb{R}^p$, $\beta>0,\gamma>0$, $t_1\geq 1$.  \\
        \For{$k = 1, 2,\cdots$}{
        $\bar{y}_k = y_k + \frac{t_k-1}{t_{k+1}}(y_k-y_{k-1})$.\\
      $\bar{\lambda}_{k+1} =\lambda_k+ \gamma t_{k+1}(Bv_{k}-b)$.\\
      $v_{k+1} ={\bf Prox}_{\frac{\beta}{t_{k+1}},g_1}\left(v_k-\frac{\beta}{t_{k+1}}(B^T\bar{\lambda}_{k+1}+\nabla g_2(\bar{y}_k)\right) $. \\
     $ y_{k+1} = \frac{1}{t_{k+1}} v_{k+1}+ \frac{t_{k+1}-1}{t_{k+1}}y_{k}$.\\
		$\lambda_{k+1} = \lambda_k + \gamma t_{k+1}(Bv_{k+1}-b)$.
}
\end{algorithm}

\begin{theorem}\label{th_A22}
Suppose that $g_1$ is a closed and $\mu_g$-strongly convex function, $g_2$ is convex has a Lipschitz continuous gradient with constant $L_{g_2}$, and the saddle points set $\Omega$ is nonempty. Let $\{y_k\}_{k \geq 1}$ be the sequence generated by Algorithm \ref{al_1B2} and  $(y^*,\lambda^*)\in\Omega$. Assume that  $\beta\leq \frac{t_1^2}{L_{g_2}+t_1^2\gamma\|B\|^2}$ and $\{t_k\}_{k\geq 1}\subseteq [1,+\infty)$ is a nondecreasing sequence satisfying
\begin{eqnarray*}
	t_{k+1}^2\leq t^2_k + t_{k+1}\quad\text{ and }\quad t^2_{k+1}\leq t^2_{k}+\beta\mu_g t_{k}.
\end{eqnarray*}
Then the following conclusions hold:
\begin{eqnarray*}
&&  \|By_{k}-b\|\leq  \frac{C_4}{t_{k}^2},\\
&& |g(y_k)-g(y^*)|\leq \frac{\mathcal{E}_{1}+C_4\|\lambda^*\|}{t_{k}^2},
\end{eqnarray*}
where $\mathcal{E}_{1}$ and $C_4$ are defined same as in Theorem \ref{th_B22}.
\end{theorem}

\begin{remark}
	Incorporating the appropriate parameter selection, along with the use of Lemma \ref{le_appt1} and Lemma \ref{le_appt2}, it is straightforward to obtain non-ergodic $\mathcal{O}(1/k^2)$ convergence rate results for both Algorithm \ref{al_1B1} and Algorithm \ref{al_1B2}. From Algorithm \ref{al_1B2}, we can see  that when the objective function is strongly convex, allowing subproblems to be solved via the proximal operator, it enables efficient solutions to certain specific problems. In the convex case, the $\mathcal{O}(1/k^2)$ convergence rate of Algorithm \ref{al_1B1} is consistent with the convergence results of the accelerated linearized ALM in \cite{Xu2017}, which also employs Nesterov's second scheme \eqref{NestSS} within a primal-dual framework.	
	In the strongly convex case, the $\mathcal{O}(1/k^2)$ convergence rate of Algorithm \ref{al_1B2} with proximal calculation of $g_1$ is consistent with the convergence results of the accelerated linearized ALM in \cite{Sabach2022}.
	\end{remark}

\section{Connections to existing accelerated methods}\label{sec5}

In this section, we explore  the connection between our accelerated linearized ADMMs and existing accelerated methods. Building upon the discussion in Section \ref{sec_para} and the insights provided by Corollary \ref{cor_para}, it is evident that with a suitable selection of $\alpha, \beta, \gamma, t_1$, 
all the convergence results of our proposed methods are guaranteed when the sequence $t_k$ satisfies the condition:
\[t_{k+1}^2\leq t^2_k + t_{k+1} \quad \forall k\geq 1.\]

In the following, we present several examples of accelerated methods with Nesterov extrapolation and compare them with our accelerated methods. It is worth noting that various choice rules  of $t_k$ illustrated in these examples can also be used in our accelerated linearized ADMMs.

\noindent {\bf Example 1: (He-Hu-Fang scheme for  one-block problem \eqref{one_bolck})}

Under the rule
\[t_k = 1+\frac{k-2}{\alpha-1},\quad \forall k\geq 1\]
with $\alpha\geq 3$, He et al. \cite{HeNA}  proposed  the following accelerated linearized primal-dual method\footnote{By variable substitution, we eliminate the intermediate variable $\bar{\lambda}_k$. This will not have any effect on the convergence results of the paper.} for problem \eqref{one_bolck}: 
\begin{numcases}{}
         \bar{x}_k = x_k+\frac{t_k-1}{t_{k+1}}(x_k-x_{k-1}),\nonumber \\
		x_{k+1} = \mathop{\arg\min}_x f_1(x)+\langle \nabla f_2(\bar{x}_k)+A^T\lambda_k,x\rangle,\nonumber\\
		\qquad\qquad+\frac{st_{k+1}(t_{k+2}-1)}{2}\|A(x-x_k)+\frac{1}{t_{k+1}}(Ax_k-b)\|^2 + \frac{t_{k+1}}{2s(t_{k+2}-1)}\|x-\bar{x}_k\|^2,\label{Hena}\\
		u_{k+1} = x_{k+1}+(t_{k+1}-1)(x_{k+1}-x_k),\nonumber\\
	 \lambda_{k+1} = \lambda_k + s(t_{k+2}-1)(Au_{k+1}-b),\nonumber
 \end{numcases}
 and obtained non-ergodic $\mathcal{O}(1/k^2)$ convergence rates of  the objective residual and the feasibility violation, see \cite[Theorem 2]{HeNA}. This method can be equivalently expressed as 
\begin{equation*}
	\begin{cases}
		 \bar{x}_k = x_k+\frac{t_k-1}{t_{k+1}}(x_k-x_{k-1}),\\
		u_{k+1}-u_k = -{\alpha \boldsymbol{\eta_k} t_{k+1}}(\widetilde{\nabla} f_1({x_{k+1}})+\nabla f_2(\bar{x}_k)+{A^T{\lambda}_{k+1}}),\\
		u_{k+1} = x_{k+1}+(t_{k+1}-1)(x_{k+1}-x_k),\\
		\lambda_{k+1} = \lambda_k + \gamma  \boldsymbol{\eta_k} t_{k+1}(Au_{k+1}-b)
	\end{cases}
\end{equation*}
with rescaling  $\eta_k = \frac{t_{k+2}-1}{t_{k+1}}$ and suitable $\gamma$. We can verify that
 \[t_{k+1}^2-t^2_k-t_{k+1}=\frac{(3-\alpha)k-(\alpha-5/2)^2-3/4}{(\alpha-1)^2}\leq 0. \]
Hence,  algorithm \eqref{Hena} can be considered as a rescaled variant of Algorithm \ref{al_2B1}, where  the rescaling coefficient is $\eta_k$.

In the case of smooth functions, Bo{\c t} et al. \cite{BotMP2022} also considered the rule 
\[ t_{k+1}^2\leq t^2_k +\eta t_{k+1} \text{ with } \eta\leq 1\]
for the choice of $t_k$.

\begin{remark}
	 Consider the following rescaling version of scheme \eqref{disB} for problem \eqref{ques}:
\begin{subequations}\label{disA_res}
\begin{numcases}{}
        (\bar{x}_k, \bar{y}_k) = (x_k,y_k) + \frac{t_k-1}{t_{k+1}}[(x_k,y_k)-(x_{k-1},y_{k-1})],\nonumber \\
		u_{k+1}-u_k = -\boldsymbol{\alpha\eta_k t_{k+1}}(\widetilde{\nabla} f_1(\boldsymbol{x_{k+1}})+\nabla f_2(\bar{x}_k)+{A^T\bar{\lambda}_{k+1}}),\nonumber\\
		v_{k+1}-v_k = -\boldsymbol{\frac{\beta\eta_k}{t_{k+1}}}(\widetilde{\nabla} g_1(\boldsymbol{y_{k+1}})+\nabla g_2(\bar{y}_k) +{B^T{\hat{\lambda}}_{k+1}}),\nonumber\\
		(x_{k+1},y_{k+1}) = \frac{1}{t_{k+1}}(u_{k+1},v_{k+1})+ \frac{t_{k+1}-1}{t_{k+1}}(x_{k},y_k),\nonumber\\
		\lambda_{k+1} = \lambda_k + \boldsymbol{\gamma\eta_k} t_{k+1}(Au_{k+1}+Bv_{k+1}-b),\nonumber
 \end{numcases}
\end{subequations}
where $\eta_k>0$ is the rescaling coefficient with $\lim_{k\to+\infty}\eta_k=1$.
Define the  rescaling energy sequence of  $\mathcal{E}_k$ as 
\begin{eqnarray*}
	\mathcal{E}_k &=&\eta_{k-1}t^2_k(\mathcal{L}(x_k,y_k,\lambda^*)-\mathcal{L}(x^*,y^*,\lambda^*))+\frac{1}{2\alpha}\|u_k-x^*\|^2+\frac{t_{k+1}^2}{2\beta}\|v_k-x^*\|^2+\frac{1}{2\gamma}\|\lambda_k-\lambda^*\|^2.
\end{eqnarray*}
By  arguments similar as in  Lemma \ref{le_A1} and Theorem \ref{th_A1}, we can easily obtain the $\mathcal{O}\left(\frac{1}{\eta_k t_k^2}\right)$ convergence rate under the assumption
\[\eta_k t_{k+1}^2\leq \eta_{k-1}t_k^2+\eta_k t_{k+1}.\]
Since  $\lim_{k\to+\infty}\eta_k=1$, this situation is not essentially different from \eqref{disB}, and so we will not do in-depth research.
\end{remark}

\noindent {\bf Example 2: (Xu scheme for one-block problem \eqref{one_bolck})}

Under the rule
\[t_k = \frac{k}{2},\quad \forall k\geq 1,\]
which  yields $t_{k+1}^2-t^2_k-t_{k+1}=-\frac{1}{4} \leq 0$, Xu \cite{Xu2017} proposed  the following accelerated linearized ALM for  for problem \eqref{one_bolck}: 
\begin{numcases}{}
         \bar{x}_k = x_k+\frac{t_k-1}{t_{k+1}}(x_k-x_{k-1}),\nonumber \\
		u_{k+1} = \mathop{\arg\min}_u\left(f_1(u)+\langle A^T\lambda_k+\nabla f_2(\bar{x}_k),u\rangle+\frac{\gamma t_{k}}{2}\|Au-b\|^2+\frac{1}{2\alpha t_{k}}\left\|u-u_k\right\|^2\right),\nonumber\\
		x_{k+1}= \frac{1}{t_{k+1}}u_{k+1}+ \frac{t_{k+1}-1}{t_{k+1}}x_{k}, \nonumber\\
	 \lambda_{k+1} = \lambda_k + \gamma t_{k}(Au_{k+1}-b),\nonumber
 \end{numcases}
 This method can be seen as a rescaled version of Algorithm \ref{al_1B1}, where the  rescaling coefficient is $\eta_k = \frac{t_k}{t_{k+1} }$. The non-ergodic convergence rate of $\mathcal{O}(1/k^2)$ established in \cite[Theorem 2.9]{Xu2017} is consistent with the results obtained in this approach.

\noindent {\bf Example 3: (Sabach-Teboulle scheme for  one-block problem \eqref{one_bolck_two})}

Under the rule 
\[ t_{k+1}=\frac{1+\sqrt{1+4t_k^2}}{2},\]
Sabach and Teboulle \cite{Sabach2022} proposed the  unifying and simplifying faster Lagrangian-based method for problem \eqref{ques} with  $f=0$ and $A=0$:
\begin{numcases}{}
         \bar{\lambda}_k = \lambda_k+\alpha t_k(t_k-1)(By_k-b),\nonumber \\
		v_{k+1} \in \textbf{Prim}_{t_{k+1}}(v_k,\bar{\lambda}_k),  \nonumber
		\\
		y_{k+1}= \frac{1}{t_{k+1}}v_{k+1}+ \frac{t_{k+1}-1}{t_{k+1}}y_{k},\nonumber\\
	 \lambda_{k+1} = \lambda_k + \gamma t_{k}(Bv_{k+1}-b),\nonumber
 \end{numcases}
where $\textbf{Prim}_{t}(\cdot)$ denotes a nice primal algorithmic map, like the $v_k$ update in Algorithm \ref{al_1A} and Algorithm \ref{al_1B}. It was proven in \cite[Theorem 4.4]{Sabach2022} that this method achieves a  non-ergodic convergence rate of $\mathcal{O}(1/k^2)$ under the strong convexity assumption of the function $g$, even when  the augmented terms are linearized. This framework aligns with our accelerated linearized ALM  based on Nesterov's second scheme, as outlined in Algorithm \ref{al_1B2}.

 \noindent {\bf Example 4: (Tseng scheme for unconstrained problem \eqref{tbl-f1})}

 Tseng \cite{Tseng2008}  considered  Nesterov's second scheme \eqref{NestSS} under the rule
\[t_{k} = \frac{k+1}{2}\]
and 
\[ t_{k+1}=\frac{1+\sqrt{1+4t_k^2}}{2}.\]
It was proven in \cite[Corollary 1]{Tseng2008} (also see \cite[Theorem 2.3]{LiML2019ACC}) that this method achieves a  non-ergodic convergence rate of $\mathcal{O}(1/k^2)$.

 \noindent {\bf Example 5: (Chambolle-Dossal scheme for unconstrained problem \eqref{tbl-f1})}

Chambolle and Dossal in \cite{Chambolle}  considered  Nesterov's second scheme \eqref{FISTA} under the rule
	\[t_k=1+\frac{k-1}{\alpha-1},\quad  \forall k\geq 1\]
with $\alpha\geq 3$. As a result,
	\begin{eqnarray*}
		t_{k+1}^2-t_k^2\leq\frac{2}{\alpha-1}t_{k+1}\leq t_{k+1}.
	\end{eqnarray*}
The non-ergodic  $\mathcal{O}(1/k^2)$ convergence rate of  Nesterov's first scheme \eqref{FISTA} was shown in \cite[Theorem 3.2]{Chambolle}.

 \noindent {\bf Example 6: (Attouch-Cabot scheme for unconstrained problem \eqref{tbl-f1})}
	
Attouch and Cabot \cite{Attouch2018} considered Nesterov's second scheme  \eqref{FISTA} under the rule
	\[t_k=\frac{k-1}{\alpha-1},\quad  \forall k\geq 1\]
	with $\alpha\geq 3$.  Hence
	\begin{eqnarray*}
		t_{k+1}^2- t_k^2-t_{k+1}=\frac{(3-\alpha)k-1}{(\alpha-1)^2}\leq 0.
	\end{eqnarray*}
They showed in \cite[Corollary 17]{Attouch2018} the non-ergodic  $\mathcal{O}(1/k^2)$ convergence rate of  Nesterov's first scheme  \eqref{FISTA}.

 \noindent {\bf Example 7: (Beck-Teboulle scheme for unconstrained problem \eqref{tbl-f1})}

By selecting 
\[ t_{k+1}=\frac{1+\sqrt{1+4t_k^2}}{2},\]
which ensure that
\[ 	t_{k+1}^2= t_k^2+ t_{k+1}.\]
This transformation effectively turns Nesterov's first scheme \eqref{FISTA} into the well-known FISTA method, originally proposed by Beck and Teboulle in  \cite{BeckIma}.

\section{Numerical experiments}
In this section, we present two numerical experiments aimed at validating the performance of the proposed algorithms. The parameter settings of all algorithms in the experiments satisfy the parameter assumptions of the theoretical convergence rates. All codes are implemented using Python 3.8 on a MacBook laptop equipped with  Intel Core i5 CPU running at 2.30GHz and 8 GB of memory.

\subsection{Elastic net problem}

Let's explore the elastic net problem:
\begin{equation*}
	\min_{x\in\mathbb{R}^n} P(x)=\|x\|_1+\frac{\mu}{2}\|x\|^2+\frac{\eta}{2}\|Mx-b\|^2,
\end{equation*}
where  $\mu,\eta > 0$. With the constraint $x = y$, the  elastic net problem can be reformulated as:
\begin{eqnarray}\label{ENP}
	\min_{x,y} &\ F(x,y)= \|y\|_1+\frac{\mu}{2}\|y\|^2+\frac{\eta}{2}\|Mx-b\|^2,  \nonumber\\
		s.t.& \ x -y =0.
\end{eqnarray}
Certainly, problem \eqref{ENP} is equivalent to problem \eqref{ques} with $g(y) = g_1(y)=\|y\|_1+\frac{\mu}{2}\|y\|^2$,  which is  $\mu$-strongly convex, $f(x)=f_1(x) = \frac{\eta}{2}\|Mx-b\|^2$ and $A=I_n, B=-I_n$. 

Let's set $\mu=0.1$, $\eta =1$, $m=500$, and $n=1000$. We generate matrix $M$ from a Gaussian distribution. The vector $\bar{x}$ consists of $50$ non-zero elements generated uniformly from $[-10, 10]$. The observed vector $b$ is obtained as $b=M\bar{x}+\omega$, where $\omega \in \mathbb{R}^m$ is Gaussian noise with variance $\sigma^2=1e-4$. We proceed to compare the performance of our accelerated linearized ADMMs, the accelerated linearized ADMM presented in \cite{Xu2017}, and the new primal-dual algorithm introduced in \cite{Tran2018}.

Regarding our proposed algorithm:
\begin{itemize}
	\item[$\bullet$] Algorithm \ref{al_2A} - Accelerated linearized ADMM-Nesterov's first scheme (I) (ALADMM-F-I).
	\item[$\bullet$] Algorithm \ref{al_2B} - Accelerated linearized ADMM-Nesterov's first scheme (II) (ALADMM-F-II).
	\item[$\bullet$] Algorithm \ref{al_1A} - Accelerated linearized ADMM-Nesterov's Second scheme (I) (ALADMM-S-I).
	\item[$\bullet$] Algorithm \ref{al_1B} - Accelerated linearized ADMM-Nesterov's Second scheme (II) (ALADMM-S-II).
\end{itemize}
We set the parameters as $\alpha=100$, $\beta=\gamma=1$, and update $t_k$ using \eqref{max_tk} with $t_0=1$. For the compared algorithms, the parameters are set as follows:
\begin{itemize}
	\item[$\bullet$] Accelerated linearized ADMM in \cite[Algorithm 2]{Xu2017} (ALADMM): $\beta_k=\gamma_k=\frac{k+1}{20}, P^k=0, Q^k=0$.
	\item[$\bullet$] The new primal-dual  algorithm in\cite[Scheme (39)]{Tran2020} (New-PDA): $\gamma=0.999$, $\rho_0^{1+}=\frac{5\Gamma}{2}$.
\end{itemize}
It's worth noting that under Assumption \ref{ass}, all of our proposed algorithms enjoy $O(1/k^2)$ convergence rates. The ALADMM achieves $O(1/k)$ convergence rate for the iterative sequence and $O(1/k^2)$ convergence rate for the ergodic iterative sequence. The New-PDA achieves the $O(1/k^2)$ convergence rate. 

Fig. \ref{fig1} displays the results regarding ``Iterations''  (the number of iterations) and ``CPU'' (CPU time in seconds) for all algorithms. Observing from Fig. \ref{fig1}, it's evident that our proposed algorithms demonstrate superior numerical performance, competing effectively with ALADMM and New-PDA across various metrics such as composite objective residual, objective residual, and feasibility violation. This is supported by the indication of fewer iterations and less CPU time.

\begin{figure}[htbp]
\centering
\subfigure[The number of iterations]{
\begin{minipage}[t]{0.48\linewidth}
\centering
\includegraphics[width=3in]{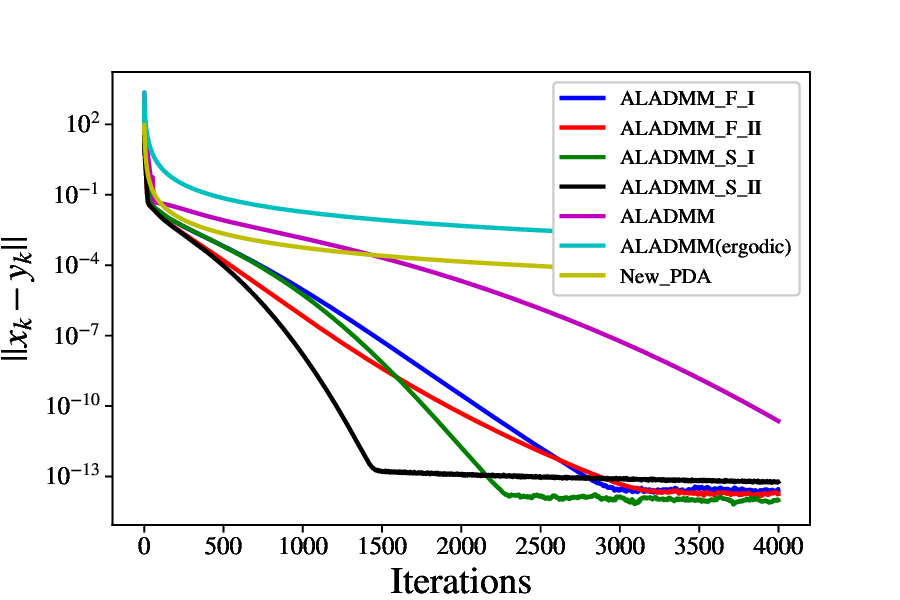}
\centering
\includegraphics[width=3in]{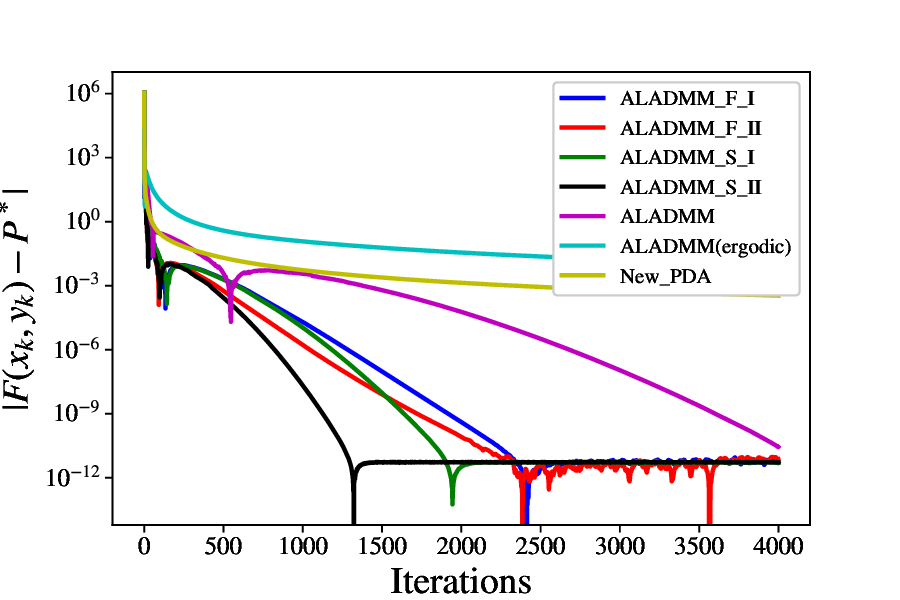}
\centering
\includegraphics[width=3in]{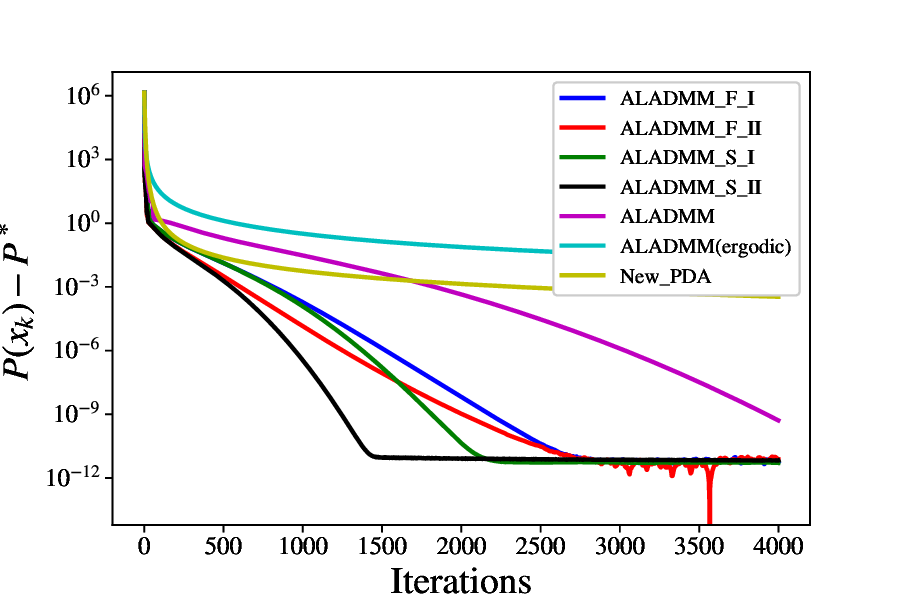}
\end{minipage}%
}%
\subfigure[CPU time in seconds]{
\begin{minipage}[t]{0.48\linewidth}
\centering
\includegraphics[width=3in]{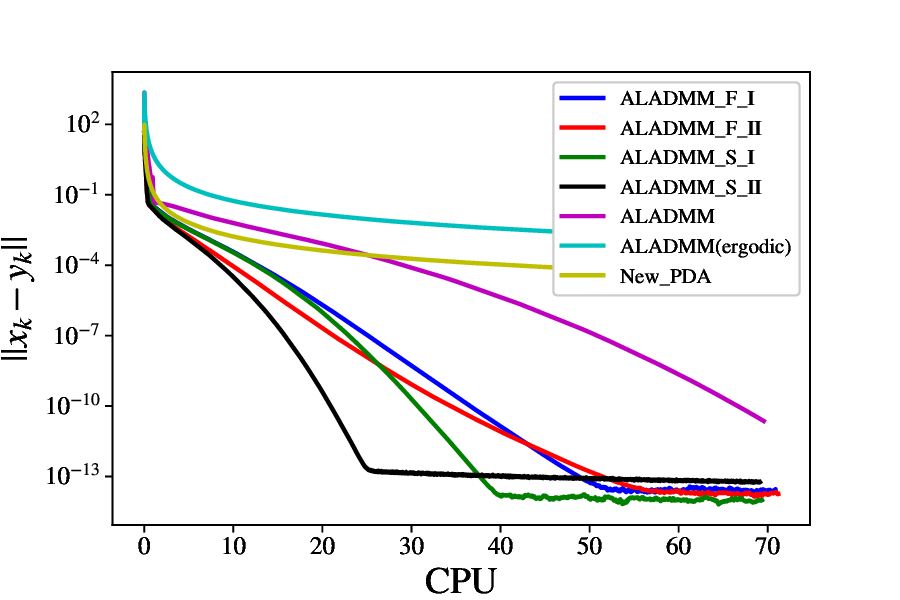}
\centering
\includegraphics[width=3in]{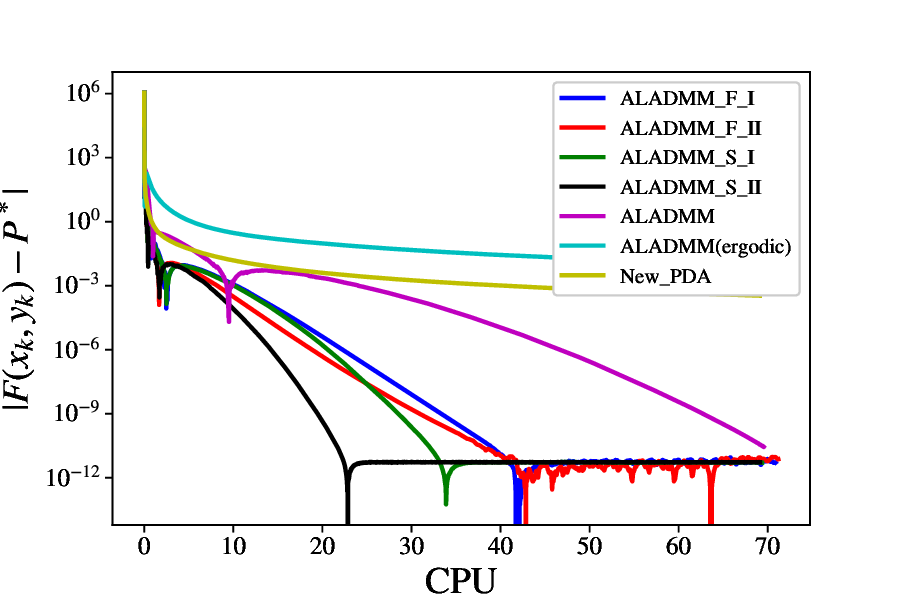}
\centering
\includegraphics[width=3in]{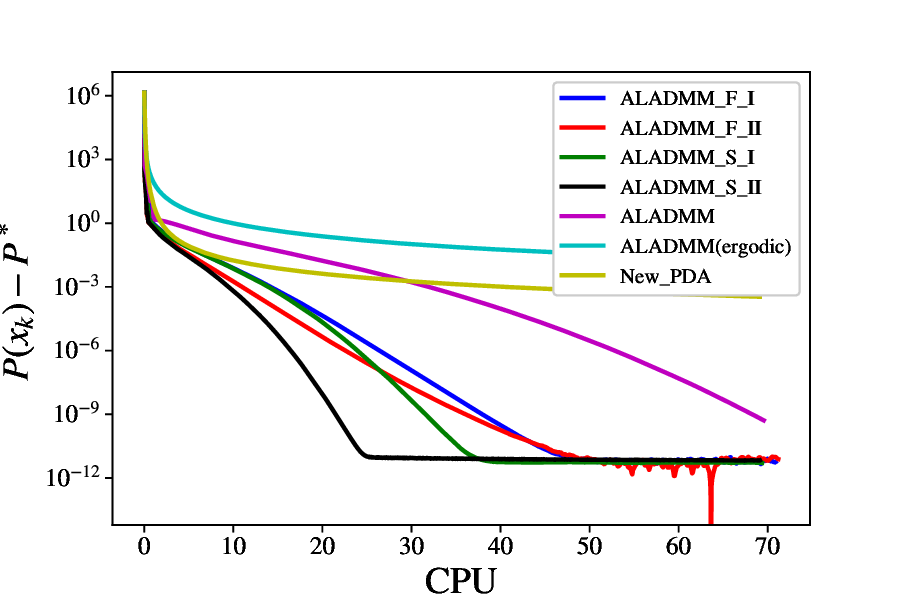}
\end{minipage}%
}
\caption{Numerical results of the proposed algorithms, ALADMM and New-PDA for problem \eqref{ENP}}\label{fig1}
\end{figure}

\subsection{Least absolute deviation regression}
	Consider the least absolute deviation (LAD) regression problem
	\begin{equation*}\label{eq_LAD}
		\min_{x\in\mathbb{R}^n}  \|x\|_1+\frac{\mu}{2}\|x\|^2 + {\eta}\|Mx-b\|_1,
	\end{equation*}
	where $A\in\mathbb{R}^{m\times n}$ and $b\in\mathbb{R}^m$ represent the given data, with $m\ll n$. It's evident that the LAD regression problem can be equivalently expressed as:
\begin{eqnarray}\label{LAD}
	\min_{x,y} &\ F(x,y)= \|y\|_1+\frac{\mu}{2}\|y\|^2+{\eta}\|x-b\|_1,  \nonumber\\
		s.t.& \ x -My =0,
\end{eqnarray}
Let's set $\mu=0.05$, $\eta =1$, $m=500$, $n=5000$. We generate matrix $M$ from a Gaussian distribution. The vector $\bar{x}$ consists of $0.1n$ non-zero elements, randomly distributed within the range $[-2, 2]$. The observed vector $b$ is obtained as $b=M\bar{x}+\omega$, where $\omega \in \mathbb{R}^m$ represents Gaussian noise with variance $\sigma^2=0.01$. We proceed by comparing the performance of our algorithms ALADMM-F-II and ALADMM-S-II (with parameters $\alpha=n, \beta =5, \gamma=\frac{1}{\beta|M|^2}$, and update $t_k$ by \eqref{max_tk}), with ALADMM from \cite{Xu2017} (with parameters $\beta_k=\gamma_k=\frac{\mu(k+1)}{2|M|^2}, P_k=0, Q_k= \frac{\mu(k+1)}{2}(I_n-\frac{M^TM}{|M|^2}$)), and New-PDA from \cite{Tran2018} (with parameters $\gamma=0.999$, $\rho_0^{1+}=\frac{5\Gamma}{2|M|^2}$). The convergence behavior of this test is depicted in Fig \ref{fig2}. Notably, our ALADMM-F-II and ALADMM-S-II algorithms outperform ALADMM and New-PDA, showcasing superior performance in both feasibility violation and objective residual.

\begin{figure}[htbp]
\centering
\subfigure[The number of iterations]{
\begin{minipage}[t]{0.48\linewidth}
\centering
\includegraphics[width=3in]{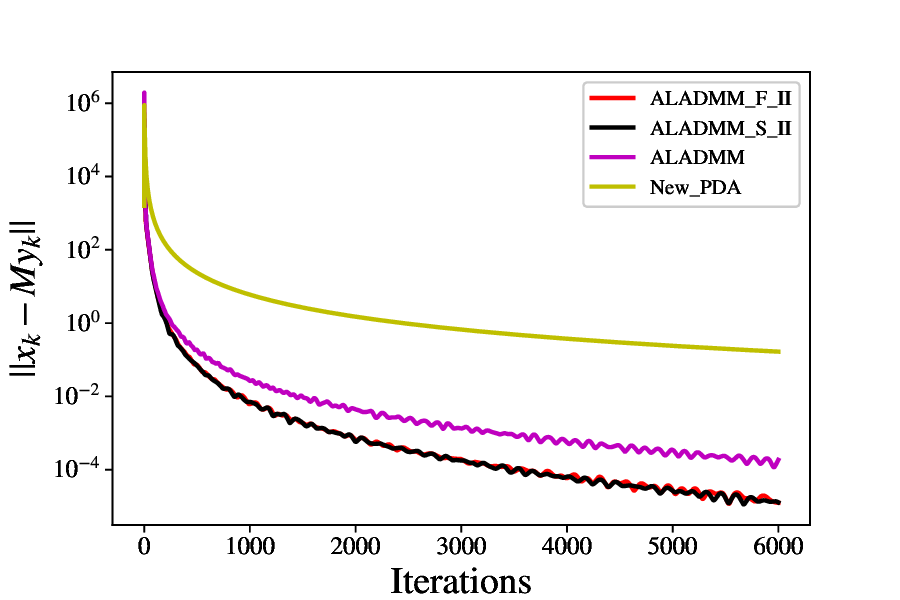}
\centering
\includegraphics[width=3in]{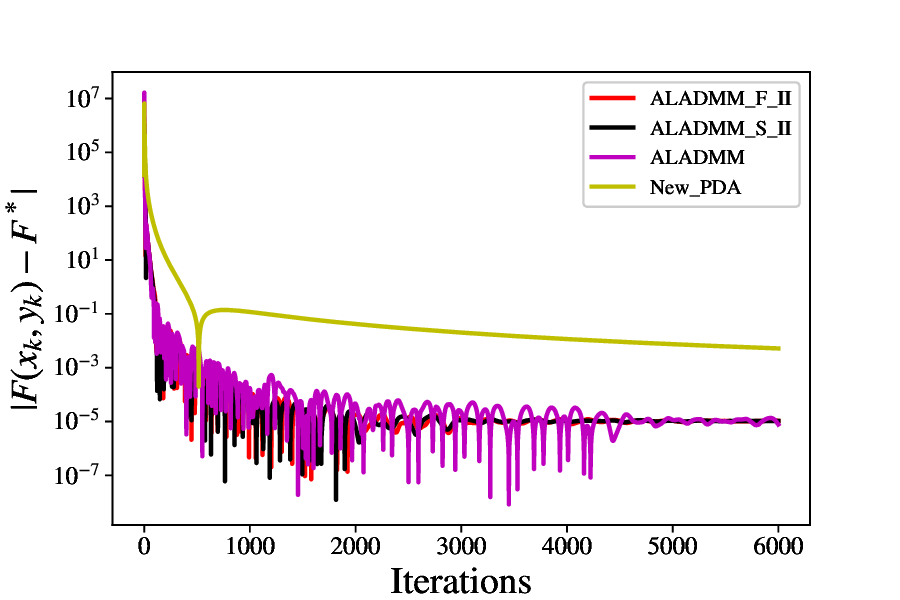}
\end{minipage}%
}%
\subfigure[CPU time in seconds]{
\begin{minipage}[t]{0.48\linewidth}
\centering
\includegraphics[width=3in]{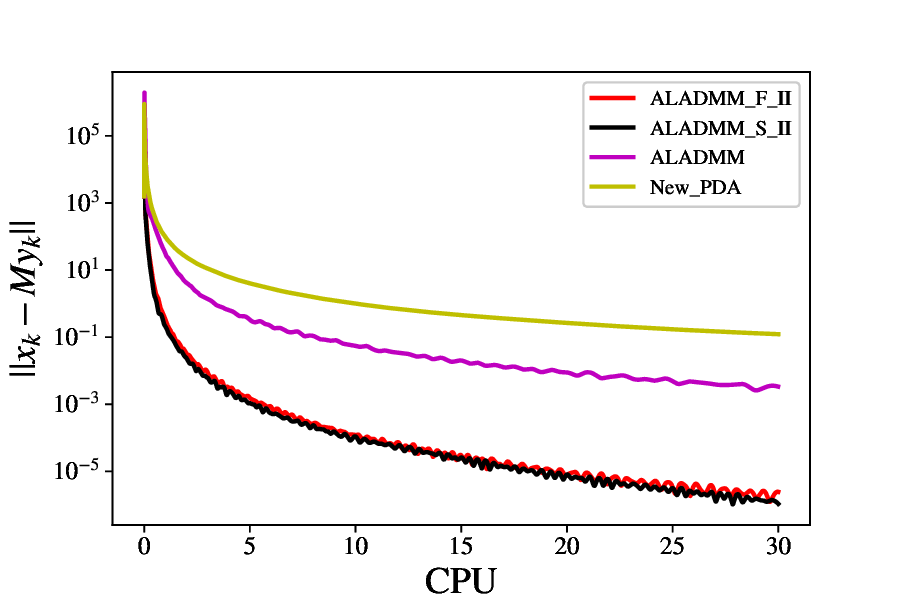}
\centering
\includegraphics[width=3in]{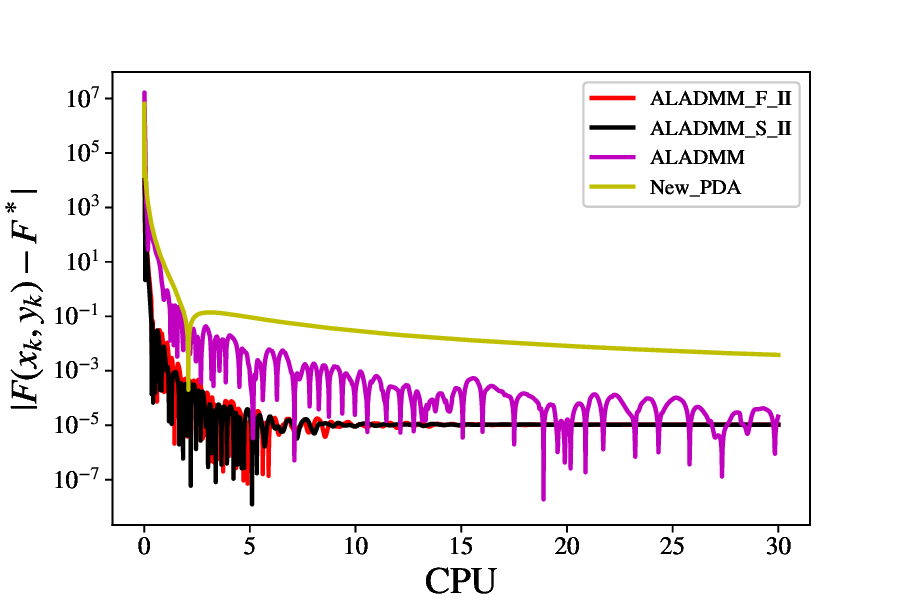}
\end{minipage}%
}
\caption{Numerical results of the proposed algorithms, ALADMM and New-PDA for problem \ref{LAD}}\label{fig2}
\end{figure}

\section{Conclusions}
This paper considers accelerated linearized ADMM methods with Nesterov extrapolation, extending two well-established Nesterov's accelerated schemes from unconstrained composite convex optimization problems to the separable convex optimization problem denoted as \eqref{ques}, which exhibits a composite structure. Assuming semi-strong convexity, the proposed methods achieve non-ergodic convergence rates of $\mathcal{O}(1/t_k^2)$ when $t_{k+1}^2\leq t_k^2 + t_{k+1}$. Selecting a classical choice for $t_k$ allows us to obtain a non-ergodic decay rate of $\mathcal{O}(1/k^2)$. Our proposed accelerated linearized ADMMs can be reduced to accelerated linearized ALMs for solving one-block linearly constrained problem \eqref{ques_one} with a decay rate of $\mathcal{O}(1/k^2)$.  Moreover, we investigate the relationships between our accelerated Linearized ADMM methods and existing accelerated ALMs and accelerated forward-backward methods. We also conduct some numerical experiments to validate the effectiveness  of the proposed algorithms.

\appendix
\section*{Appendix: Technical Lemmas}
\begin{lemma}\cite[Lemma 4]{HeAuto} \label{le_app1}
	Let $\{h_k\}_{k\geq 1}$ be a sequence of vectors in $\mathbb{R}^{n}$ and $\{a_k\}_{k\geq 1}$ be a sequence in $[0,1)$. Assume 
	\[\left\|h_{k+1}+\sum_{i=1}^{k} a_{i} h_{i}\right\|\leq C, \quad\forall k\geq 1.\]
	 Then, 
\[ \sup_{k\geq 1}\|h_k\|\leq \|h_1\|+2C.\]
\end{lemma}

\begin{lemma}\label{le_appt1}
	The positive sequence $\{t_k\}_{k\geq 1}$ generated by $t_{k+1} = \frac{1+\sqrt{1+4t_k^2}}{2}$ with $t_1>0$ satisfies $t_{k}\geq t_1+\frac{k-1}{2}$.
\end{lemma}
\begin{proof}
	We can verify that $t_{k+1}^2= t^2_k + t_{k+1}$. It implies
\[t_{k+1}^2-\frac{t_{k+1}}{2}= t^2_k + \frac{t_{k+1}}{2}\geq  t^2_k + \frac{t_{k}}{2}, \]
and then 
\begin{equation*}
	t_{k+1} \geq t_k+\frac{1}{2}\geq t_1+\frac{k}{2}.
\end{equation*}
\end{proof}

\begin{lemma}\label{le_appt2}
The positive sequence $\{t_k\}_{k\geq 1}$ generated by $t_{k+1} = \sqrt{t_k^2+at_k}$ with $a > 0$ and $t_1> 0$ satisfies $t_{k}\geq t_1+{b(k-1)}$, where $b=\frac{2at_1}{a+4t_1}< \frac{a}{2}$.
\end{lemma}	
\begin{proof}
We can verify that
\begin{equation}\label{eq_app_tk}
	t_{k+1}\leq\sqrt{t_k^2+at_k+\frac{a^2}{4}}=  t_k+\frac{a}{2}
\end{equation}
and $t_{k+1}^2= t^2_k + at_{k}$. Since $a>0$ and $b=\frac{2at_1}{a+4t_1}$,  we can verify $a-2b>0 $ and $t_k\geq t_1$ for all $k\geq 1$. It follows that
\begin{eqnarray*}
	t_{k+1}^2-{bt_{k+1}} &= & t^2_k + {bt_{k}}+{(a-b)t_k} - {bt_{k+1}}\\
	&\overset{\eqref{eq_app_tk}}{\geq}&t^2_k + {bt_{k}}+{(a-b)t_k} - {b}\left(t_{k}+\frac{a}{2}\right)\\
	&= &t^2_k + {bt_{k}}+\left((a-2b)t_k-\frac{ab}{2}\right)\\
	&\geq &t^2_k + \frac{bt_{k}}{2}+(a-2b)t_1-\frac{ab}{2}\\
	&= &t^2_k + {bt_{k}},
\end{eqnarray*}
which yields $t_{k+1}\geq t_k+{b}\geq t_1+b(k-1)$.
\end{proof}

%\begin{acknowledgements}
%\end{acknowledgements}


\section*{Data Availability}
Data will be made available on reasonable request.

\section*{Conflict of Interest} The authors have no relevant financial or non-financial interests to disclose, and have no
conflict of interest to declare that are relevant to the content of this article.


%\begin{thebibliography}{}
%\end{thebibliography}

\end{document}